\documentclass[preprint,11pt]{elsarticle}

\usepackage{graphicx}
\usepackage{pifont,latexsym,ifthen,amsthm,rotating,calc,textcase,booktabs}
\usepackage{amsfonts,amssymb,amsbsy,amsmath}
\usepackage{enumitem}
\newtheorem{theorem}{Theorem}
\newtheorem{lemma}{Lemma}
\newtheorem{corollary}{Corollary}
\newtheorem{remark}{Remark}
\newtheorem{proposition}{Proposition}

\usepackage{natbib}
\newtheorem{rhp}{RH Problem}

\numberwithin{equation}{section}
\usepackage{todonotes}
\usepackage{float}
\usepackage{subfigure}
\usepackage{hyperref}
\hypersetup{
    colorlinks=true, %set true if you want colored links
    linktoc=all,     %set to all if you want both sections and subsections linked
    linkcolor=blue,  %choose some color if you want links to stand out
    citecolor=blue%choose some color if you want links to stand out
}

\begin{document}

\begin{frontmatter}
\title{   Existence of global solutions for the nonlocal derivation nonlinear Schr\"{o}dinger equation by the inverse scattering transform method}

\author[inst2]{Yuan Li}
\author[inst2]{Xinhan Liu}
\author[inst2]{Engui Fan$^{*,}$  }

\address[inst2]{ School of Mathematical Sciences and Key Laboratory of Mathematics for Nonlinear Science, Fudan University, Shanghai, 200433, China\\
* Corresponding author and e-mail address: faneg@fudan.edu.cn  }

\begin{abstract}
  We address the existence of global solutions to the initial value problem for the integrable nonlocal derivative nonlinear Schr\"{o}dinger equation in weighted Sobolev space $H^{2}(\mathbb{R})\cap H^{1,1}(\mathbb{R})$.
The key to prove this result is  to establish a bijectivity between potential and reflection coefficient  by  using the inverse scattering transform method in the form of the Riemann-Hilbert problem.

\end{abstract}

\end{frontmatter}
\tableofcontents
    % \baselineskip=18pt
% \maketitle

\section{Introduction}
% \hspace*{\parindent}

It is well known that the derivative nonlinear Schr\"{o}dinger (DNLS) equation \cite{Kaup}
 \begin{equation}\label{DNLS}
u_t(x,t)=iu_{xx}(x,t)+\sigma(u^2(x,t)\overline{u}(x,t))_x,\ \sigma=\pm 1
\end{equation}
is one of  important integrable systems.
In particular, the problem of global well-posedness for the DNLS equation \eqref{DNLS} has been extensively studied in the last two decades. For example, Hayashi, Ozawa \cite{Hay} and Wu, Guo \cite{Wu1,Wu2,Guo}
proved that the global solutions to the DNLS equation exists in Sobolev space if the initial data satisfy the smallness condition. The inverse scattering transform (IST) method to global well-posedness of the DNLS equation has advanced considerably in recent years.
Jenkins and Liu et al. \cite{Liu1,Jen1,Jen2,Jen3}
have used this techniques to prove the global well-posedness of the DNLS equation for any initial data in a weighted Sobolev space. Pelinovsky and Shimabukoro \cite{Pel,Pel1} constructed a unique global solution of the DNLS equation in a different weighted Sobolev space by using the IST method and B\"{a}cklund transformation.
  Bahouri and Perelman \cite{Bah1,Bah2} proved the global well-posedness of the DNLS equation in a very weak space with the help of the IST techniques. Their work has almost completely settled the problem of the global well-posedness of the DNLS equation.

In order to extend the equation (\ref{DNLS}) into a nonlocal case,  considering  a  linear system \cite{Abl}
\begin{equation}\label{1.9}
\Phi_x=-ik^2\sigma_3\Phi+kQ(u)\Phi,
\end{equation}
\begin{equation}\label{1.10}
\Phi_t=-2ik^{4}\sigma_{3}\Phi+\tilde{Q}(u)\Phi,
\end{equation}
where
\begin{align}
&\sigma_3=\begin{bmatrix}
    1&0\\
    0&-1
    \end{bmatrix}, \ \ \ \
Q(u)=\begin{bmatrix}
    0&u(x,t)\\
    v(x,t)&0
    \end{bmatrix},\nonumber\\
&\tilde{Q}(u)=\begin{bmatrix}
k^2uv&-i(2ik^3u-ku_x+iku^2v)\\
-i(2ik^3v+kv_x+ikuv^2)&-k^2uv
\end{bmatrix},\nonumber
\end{align}
Ablowitz   et al.  recently  found that under the symmetry reduction
$
v(x,t)=\sigma\overline{u}(x,t),\ \  \sigma=\pm1,
$
the compatibility condition of system \eqref{1.9} and \eqref{1.10} leads to the   DNLS equation (\ref{DNLS}); 
under the symmetry reduction
\begin{equation}\label{sr}
v(x,t)=i\sigma\overline{u}(-x,t),\ \  \sigma=\pm1,
\end{equation}
the compatibility condition of system \eqref{1.9} and \eqref{1.10} yields the new nonlocal  derivation nonlinear Schr\"{o}dinger (nDNLS) equation
\begin{align}
&u_t(x,t)=iu_{xx}(x,t)+i\sigma(u^2(x,t)\overline{u}(-x,t))_x,\ \  x\in\mathbb{R},\label{1.1}
\end{align}
where   $\sigma=\pm1,$  $u(x,t)$ is a complex valued function and $\overline{u}$ represents the conjugate complex number of $u$.
Unlike the well-known rich results on  the DNLS equation, there still a little work  related to the nDNLS equation \eqref{1.1}.  Ablowitz   et al.    constructed the   inverse scattering transform 
for the equation \eqref{1.1}  \cite{Abl}.
All possible nonlocal versions of the DNLS equations were derived by the nonlocal reduction from the Chen-Lee-Liu equation, the Kaup-Newell equation and the Gerdjikov-Ivanov equation  \cite{Shi}.
Zhou introduced a nonlocal version
of the conventional DNLS equation and derived the explicit expressions of solutions by Darboux transformations \cite{Zhou} .
Chen   et al.  showed the global existence and uniqueness of the mild solution in super-critical function spaces for a Chen-Lee-Liu version of the nonlocal DNLS equation   \cite{Chen}.
However  the existence of global classical solutions to the nDNLS equation \eqref{1.1} is still unknown.
The essential difficulty by using analytical method
to prove this result  is that   the conservation law and the integrable structures cannot formulate useful norm structure for the   nDNLS equation \eqref{1.1}  compared to DNLS equation \eqref{DNLS}.

In this paper we  consider  the Cauchy problem for the  nDNLS  equation \eqref{1.1} with weighted Sobolev initial data
\begin{align}
&u(x,0)=u_0(x) \in H^{2}(\mathbb{R})\cap H^{1,1}(\mathbb{R}).\label{1.2}
\end{align}
We   prove the existence of  global solutions to the Cauchy problem  \eqref{1.1}-\eqref{1.2}    under    the appropriate small  initial data condition (see \eqref{small+} below).
  The proof is achieved by using the IST method in the form of the Riemann-Hilbert (RH) problem. By using IST method,
we obtain a strong solution for nDNLS equation \eqref{1.1} compared to the mild weak solution in integral form for a Chen-Lee-Liu version of the nonlocal DNLS equation obtained in \cite{Chen}.
Our main results are as follows 

\begin{theorem}\label{t1}
For any initial date $u_{0}\in H^{2}(\mathbb{R})\cap H^{1,1}(\mathbb{R})$ and satisfies the small norm restriction \eqref{small+},
then there exists a unique global solution $u \in C(\mathbb{R}, H^{2}(\mathbb{R})\cap H^{1,1}(\mathbb{R}))$ to the Cauchy problem \eqref{1.1}-\eqref{1.2}. Moreover, the map $u_{0}\mapsto u $ is Lipschitz continuous from $H^{2}(\mathbb{R})\cap H^{1,1}(\mathbb{R})$ to $C(\mathbb{R}, H^{2}(\mathbb{R})\cap H^{1,1}(\mathbb{R}))$.
\end{theorem}

\begin{remark}  Here we remark that compared to that of the DNLS equation \cite{Pel}, 
  the Jost function and  scattering data for the nDNLS equation \eqref{1.1}  is short of   symmetry properties, which 
This distinctive feature not only affects the solvability of the associated RH problem for the Cauchy problem \eqref{1.1}-\eqref{1.2}, but also makes more   analytical   difficult.
To ensure  the existence of the  solution to the RH problem, a small norm assumption \eqref{small+} on the initial data is necessary.
\end{remark}

The structure of the paper is as follows. In Section 2, we first establish the Jost functions in the $z$-plane and give its existence, asymptotics and regularity. Then, we returned to the $k$-plane, constructed the scattering data and gave the corresponding regularity.
Section 3 constructs several RH problems and defines the corresponding reflection coefficients. In order to ensure the solvability of the RH problems, a suitable small norm condition is found. Further, the regularity of the reflection coefficients and the Lipshitz continuity with respect to the potential are also given.
Section 4 proves the solvability of the RH problem with the help of properties of the Cauchy operator and gives an estimate of its solution. The Reconstruction formulas for the potential and the corresponding estimate are also obtained.
Finally, we give the time evolution of the linear equation \eqref{1.10} and the proof of the Theorem \ref{t1} in Section 5.

\section{Direct scattering transform}

In this section, we  state some  main
results on  the  direct  scattering transform associated with the Cauchy problem
\eqref{1.1}-\eqref{1.2}.

\subsection{Notations}

We introduce some notations used in this paper. The Sobolev space
\begin{align*}
H^{m}(\mathbb{R})=\{f\in L^{2}(\mathbb{R}): \partial_{x}^{j}f\in L^{2}(\mathbb{R}), \ j=1,2,\cdots,m\},
\end{align*}
 equipped with the norm
\begin{align*}
\|f\|_{H^{m}(\mathbb{R})}=\sum_{j=0}^{m}\|\partial_{x}^{j}f\|_{L^{2}(\mathbb{R})}.
\end{align*}
A weighted space $L^{p,s}(\mathbb{R})$  is defined by
\begin{align*}
L^{p,s}(\mathbb{R})=\{f\in L^{p}(\mathbb{R}): \langle \cdot\rangle^{s} f\in L^{p}(\mathbb{R})\},\ \ s\geq0,
\end{align*}
equipped with the norm
\begin{align*}
\|f\|_{L^{p,s}(\mathbb{R})}=\|\langle \cdot\rangle^{s} f\|_{L^{p}(\mathbb{R})},
\end{align*}
where $\langle \cdot \rangle=\sqrt{1+|\cdot|^2}$. The weighted Sobolev space $H^{m,s}(\mathbb{R})$ is defined by
\begin{align*}
H^{m,s}(\mathbb{R})=\{f\in L^{2}(\mathbb{R}): \langle \cdot\rangle^{s}\partial_{x}^{j}f\in L^{2}(\mathbb{R}),\ j=1,2,\cdots,m\},\ \ s\geq0,
\end{align*}
equipped with the norm
\begin{align*}
\|f\|_{H^{m,s}(\mathbb{R})}=\sum_{j=1}^{m}\|\langle \cdot\rangle^{s}\partial_{x}^{j}f\|_{L^{2}(\mathbb{R})},\ s\geq0.
\end{align*}

\subsection{The Jost functions in the $z$-variable}

We consider the case that the compatibility condition of system \eqref{1.9} and \eqref{1.10} under the symmetry reduction \eqref{sr} leads to the nDNLS equation \eqref{1.1}. In fact, due to the invariance of $x\rightarrow -x$, the sign of $\sigma=\pm1$ does not matter in the DNLS type equations, so it is sufficient to analyse only for $\sigma=1$.
According to the inverse scattering theory, the time dependence of the function is ignored for the moment and the time variable $t$ is seen as fixed.

Note that the equation \eqref{1.9}, unlike the nonlocal nonlinear Schr\"{o}dinger equation, the standard fixed point argument for Volterra's integral equations associated with the linear equation \eqref{1.9} is not uniform in $k$ as $|k|\rightarrow\infty$ if $Q\in L^{1}(\mathbb{R})$.
Therefore, we solve this problem below by transforming the linear equation \eqref{1.9}. For any $u(x)\in L^{\infty}(\mathbb{R})$, $k\in \mathbb{C}$, we define the following two transformations
\begin{align}
\Phi_1(x,k)=T_1(x,k)\Phi(x,k),\ \ \Phi_2(x,k)=T_2(x,k)\Phi(x,k), \label{2.3}
\end{align}
where
\begin{align}
T_1(x,k)=\begin{bmatrix}1&0\\v(x)&2ik\end{bmatrix},\ \ T_2(x,k)=\begin{bmatrix}2ik&-u(x)\\0&1\end{bmatrix}.
\end{align}
The transformation  (\ref{2.3}) changes   the spectral problem \eqref{1.9} into the following two Zakharov-Shabat
type spectral problems:
{\small\begin{align}\label{2.4}
\partial_x\Phi_{1}=-ik^2\Phi_1+Q_1(u)\Phi_1,\ \ Q_1(u)=\frac{1}{2i}\begin{bmatrix}-u(x)v(x)&u(x)\\2iv_x(x)-u(x)v^2(x)&u(x)v(x)\end{bmatrix},
\end{align}}
and
{\small\begin{align}\label{2.5}
\partial_x\Phi_{2}=-ik^2\Phi_2+Q_2(u)\Phi_2,\ \ Q_2(u)=\frac{1}{2i}\begin{bmatrix}-u(x)v(x)&-2iu_x(x)-u^2(x)v(x)\\v(x)&u(x)v(x)
\end{bmatrix}.
\end{align}}

Let $\Phi_1^{\pm}(x,k), \Phi_2^{\pm}(x,k)$ be even vectorial Jost function solutions satisfying the spectral problems \eqref{2.4} and \eqref{2.5} respectively, and have the following asymptotic properties
\begin{align*}
\Phi_1^{\pm}(x,k)\rightarrow e^{-ik^{2}x}, \ \ \Phi_2^{\pm}(x,k)\rightarrow e^{ik^{2}x},\ \ x\rightarrow \pm\infty.
\end{align*}
It is more convenient to define a new complex variable $z=k^2$ and work with the normalized Jost solutions
\begin{align}\label{2.6}
\mu_{\pm}(x,z)=\Phi_1(x,k)e^{izx}, \ \ \nu_{\pm}(x,z)=\Phi_2(x,k)e^{-izx},
\end{align}
with the asymptotic behavior
\begin{align}\label{2.7}
\mu_{\pm}(x,z)\rightarrow e_1, \ \ \nu_{\pm}(x,z)\rightarrow e_2,\ \ x\rightarrow \pm\infty.
\end{align}
Here $e_{1}=[1,0]^{T}$ and $e_{2}=[0,1]^{T}$.
The normalized Jost solutions have the following integral expression
\begin{align}\label{2.8}
\mu_{\pm}(x,z)=e_1+\int_{\pm\infty}^{x}\begin{bmatrix}1&0\\0&e^{2iz(x-y)}\end{bmatrix}Q_1(u(y))\mu_{\pm}(y,z)dy,
\end{align}
\begin{align}\label{2.9}
\nu_{\pm}(x,z)=e_2+\int_{\pm\infty}^{x}\begin{bmatrix}e^{2iz(x-y)}&0\\0&1\end{bmatrix}Q_2(u(y))\nu_{\pm}(y,z)dy.
\end{align}

The following lemmas will introduce the existence, uniqueness, analyticity, asymptotics, smoothness, and Lipschitz continuity of the Jost functions $\mu_{\pm}(x,z)$ and $\nu_{\pm}(x,z)$ satisfying the equations \eqref{2.8} and \eqref{2.9} respectively.
\begin{lemma}\label{l1}
If $u\in H^{1,1}(\mathbb{R})$, then for every $z\in\mathbb{R}$, integral equations \eqref{2.8} and \eqref{2.9} there exists unique solutions $\mu_{\pm}(x,z)\in L_{x}^{\infty}(\mathbb{R})$ and $\nu_{\pm}(x,z)\in L_{x}^{\infty}(\mathbb{R})$ respectively, and
\begin{equation}\label{2.10}
\|\mu_{\mp}(x,z)\|_{L_{x}^{\infty}}+\|\nu_{\pm}(x,z)\|_{L_{x}^{\infty}}\leq C,\ \ z\in\mathbb{C}^{\pm},
\end{equation}
where $C$ is a constant independent of $z$.
Moreover, $\forall x\in\mathbb{R}$, $\mu_{-}(x,z), \nu_{+}(x,z)$ are continued analytically in $z\in\mathbb{C}^+$, while $\mu_{+}(x,z), \nu_{-}(x,z)$ are continued analytically in $z\in\mathbb{C}^-$.
\end{lemma}
\begin{proof}
We provide a detailed proof using $\mu_{-}(x,z)$ as an example, and the proofs for the other Jost functions are similar. We denote the $L^1$ matrix norm of the $2\times2$ matrix function $Q$ as
\begin{align*}
\|Q\|_{L^{1}(\mathbb{R})}:=\sum_{i,j=1}^{2}\|Q_{i,j}\|_{L^1}.
\end{align*}
If $u\in H^{1,1}(\mathbb{R})$, then $Q_{1}\in L^{1}(\mathbb{R})$. Define Neumann series
\begin{align*}
\omega_{0}(x,z)=e_{1},\ \ \omega_{n+1}(x,z)=\int_{-\infty}^{x}F(x,y,z)\omega_{n}(y)dy,
\end{align*}
where $F(x,y,z)=\text{diag}[1,e^{2iz(x-y)}]Q_{1}(y)$.
For every $\text{Im}z\geq 0$ and for every $x\in\mathbb{R}$, we have
\begin{align*}
|\omega_{1}(x,z)|\leq\int_{-\infty}^{x}|F(x,y,z)\omega_{0}(y)|dy\leq\int_{-\infty}^{x}|Q_{1}(y)|dy\triangleq \rho(x),
\end{align*}
and $\rho(x)\leq \|Q_{1}\|_{L^{1}}$, $\rho_{x}(x)=|Q_{1}(x)|$. Further we have
\begin{align*}
|\omega_{2}(x,z)|\leq\int_{-\infty}^{x}|Q_{1}(y)\omega_{1}(y)|dy\leq\int_{-\infty}^{x}\rho_{y}(y)\rho(y)dy\leq\frac{\left\|Q_1\right\|_{L^{1}}^{2}}{2}.
\end{align*}
Using mathematical induction we get
\begin{align*}
|\omega_{n}(x,z)|\leq\frac{\left\|Q_1\right\|_{L^{1}}^{n}}{n!},
\end{align*}
therefore
\begin{align}\label{2.13}
|\sum_{n=0}^{\infty}\omega_{n}(x,z)|\leq\sum_{n=0}^{\infty}\frac{\left\|Q_1\right\|_{L^{1}}^{n}}{n!}=e^{\left\|Q_1\right\|_{L^{1}}^{n}}.
\end{align}
We know by direct verification that $\sum_{n=0}^{\infty}\omega_{n}(x,z)\triangleq \mu_{-}(x,z)$ satisfies \eqref{2.8} and that $\mu_{-}(x,z)$ is a unique solution of \eqref{2.8} by Gronwall's inequality.
In addition, uniform boundedness \eqref{2.10}, as well as analyticity, is easily obtained from \eqref{2.13}.
\end{proof}
\begin{lemma}\label{l2}
If $u\in H^{1,1}(\mathbb{R})$, then for every $x\in\mathbb{R}$, we have
\begin{align}
&\lim_{|z|\rightarrow\infty}\mu_{\pm}(x,z)=\mu_{\pm}^{\infty}(x)e_1,\ \ \mu_{\pm}^{\infty}(x):=e^{-\frac{1}{2i}\int_{\pm\infty}^xu(y)v(y)dy},\label{2.15}\\
&\lim_{|z|\rightarrow\infty}\nu_{\pm}(x,z)=\nu_{\pm}^{\infty}(x)e_2,\ \ \nu_{\pm}^{\infty}(x):=e^{\frac{1}{2i}\int_{\pm\infty}^xu(y)v(y)dy}.\label{2.16}
\end{align}
Moreover, if $u\in C^1(\mathbb{R})\cap H^{1,1}(\mathbb{R})$, then for every $x\in\mathbb{R}$, we have
\begin{align}
&\lim_{|z|\rightarrow\infty}z[\mu_{\pm}(x,z)-\mu_{\pm}^{\infty}(x)e_1]=\alpha_{\pm}(x),\label{2.17}\\
&\lim_{|z|\rightarrow\infty}z[\nu_{\pm}(x,z)-\nu_{\pm}^{\infty}(x)e_2]=\beta_{\pm}(x),\label{2.18}
\end{align}
where $\alpha_{\pm}(x):=[\alpha_{\pm}^{(1)}(x),\alpha_{\pm}^{(2)}(x)]^{T},\ \ \beta_{\pm}(x):=[\beta_{\pm}^{(1)}(x),\beta_{\pm}^{(2)}(x)]^{T}$ and
\begin{align*}
&\alpha_{\pm}^{(1)}(x)=\frac{1}{4}e^{-\frac{1}{2i}\int_{\pm\infty}^{x}u(y)v(y)dy}\int_{\pm\infty}^{x}[u(y)v_y(y)-\frac{1}{2i}u^2(y)v^2(y)]dy,\\
&\alpha_{\pm}^{(2)}(x)=-\frac{1}{2i}\partial_x[v(x)e^{-\frac{1}{2i}\int_{\pm\infty}^{x}u(y)v(y)dy}],\\
&\beta_{\pm}^{(1)}(x):=-\frac{1}{2i}\partial_x[u(x)e^{\frac{1}{2i}\int_{\pm\infty}^{x}u(y)v(y)dy}],\\
&\beta_{\pm}^{(2)}(x):=-\frac{1}{4}e^{\frac{1}{2i}\int_{\pm\infty}^{x}u(y)v(y)dy}\int_{\pm\infty}^{x}[u_y(y)v(y)+\frac{1}{2i}u^2(y)v^2(y)]dy.
\end{align*}
\end{lemma}
\begin{proof}
We still only give the proof of $\mu_{-}(x,z)$ and the rest is similar. Note that the limit here is along a contour that is in the analytical domain of the Jost functions and makes $|\text{Im}(z)|\rightarrow\infty$. Let $\mu_{-}(x,z)=[\mu_{-}^{(1)}(x,z),\mu_{-}^{(2)}(x,z)]^{T}$, and according to the integral equation \eqref{2.8} we get
\begin{align}
&\mu_{-}^{(1)}(x,z)=1-\frac{1}{2i}\int_{-\infty}^{x}u(y)v(y)\mu_{-}^{(1)}(y,z)-u(y)\mu_{-}^{(2)}(y,z)dy,\label{2.19}\\
&\mu_{-}^{(2)}(x,z)=\frac{1}{2i}\int_{-\infty}^{x}e^{2iz(x-y)}\varphi(y,z)dy,\label{2.20}
\end{align}
where
\begin{align*}
\varphi(x,z)=[2i\partial_{x}v(x)-u(x)v^{2}(x)]\mu_{-}^{(1)}(x,z)+u(x)v(x)\mu_{-}^{(2)}(x,z).
\end{align*}
When $\text{Im}(z)>0$, Lebesgue's dominated convergence theorem can be applied to \eqref{2.20} due to $u\in H^{1,1}(\mathbb{R})$ and uniform boundedness \eqref{2.10} holds, yielding
\begin{align}\label{2.20a}
\lim_{|z|\rightarrow\infty}\mu_{-}^{(2)}(x,z)=0.
\end{align}
Substituting \eqref{2.20a} into \eqref{2.19} gives the limit \eqref{2.15}.

In order to obtain the formula \eqref{2.17}, \eqref{2.20} is rewritten in the following form
\begin{align*}
\mu_{-}^{(2)}(x,z)=&\frac{1}{2i}\int_{-\infty}^{x-\delta}e^{2iz(x-y)}\varphi(y,z)dy+\frac{\varphi(x,z)}{2i}\int_{x-\delta}^{x}e^{2iz(x-y)}dy\\
&+\frac{1}{2i}\int_{x-\delta}^{x}e^{2i(x-y)}[\varphi(y,z)-\varphi(x,z)]dy\\
\triangleq & h_{1}(x,z)+h_{2}(x,z)+h_{3}(x,z).
\end{align*}
Since $u\in C^1(\mathbb{R})\cap H^{1,1}(\mathbb{R})$, we have $\varphi(\cdot,z)\in C(\mathbb{R})\cap L^{1}(\mathbb{R})$, $\forall$ $\text{Im}(z)>0$. If we let $\delta=[\text{Im}(z)]^{-1/2}$, then
\begin{align*}
\lim_{|z|\rightarrow\infty}z\mu_{-}^{(2)}(x,z)&=\lim_{|z|\rightarrow\infty}z h_{2}(x,z)=\frac{1}{4}\lim_{|z|\rightarrow\infty}\varphi(x,z)\\
&=\frac{1}{4}[2i\partial_{x}v(x)-v^{2}(y)u(y)]\mu_{-}^{\infty}(x)\\
&=\alpha_{-}^{(2)}(x).
\end{align*}
Taking the derivative of the equation \eqref{2.19} with respect to the $x$ variable and using $\overline{\mu}_{-}^{\infty}(x)$ as the integrating factor, we get
\begin{equation*}
\mu_{-}^{(1)}(x,z)=\mu_{-}^{\infty}(x)+\frac{1}{2i}\mu_{-}^{\infty}(x)\int_{-\infty}^{x}u(y)\overline{\mu}_{-}^{\infty}(y)\mu_{-}^{(2)}(y,z)dy,
\end{equation*}
and hence
\begin{align*}
\lim_{|z|\rightarrow\infty}z[\mu_{-}^{(1)}(x,z)-\mu_{-}^{\infty}(x)]&=\frac{1}{2i}\mu_{-}^{\infty}(x)\lim_{|z|\rightarrow\infty}\int_{-\infty}^{x}u(y)\overline{\mu}_{-}^{\infty}(y)z\mu_{-}^{(2)}(y,z)dy\\
&=\alpha_{-}^{(1)}(x).
\end{align*}
\end{proof}
\begin{lemma}\label{l3}
Suppose that $u\in H^{1,1}(\mathbb{R})$, then we have
\begin{align}\label{2.30}
\mu_{\pm}(x,z)-\mu_{\pm}^{\infty}(x)e_1, \ \nu_{\pm}(x,z)-\nu_{\pm}^{\infty}(x)e_2\in L_{x}^{\infty}(\mathbb{R}^{\pm}, H_{z}^1(\mathbb{R})),
\end{align}
and the map
\begin{align}\label{2.43}
u\mapsto[\mu_{\pm}(x,z)-\mu_{\pm}^{\infty}(x)e_1, \nu_{\pm}(x,z)-\nu_{\pm}^{\infty}(x)e_2]
\end{align}
is Lipschitz continuous from $H^{1,1}(\mathbb{R})$ to $L_{x}^{\infty}(\mathbb{R}^{\pm}, H_{z}^{1}(\mathbb{R}))$.

Moreover, if $u\in H^2(\mathbb{R})\cap H^{1,1}(\mathbb{R})$, then we have
\begin{equation}\label{2.31}
z(\mu_{\pm}(x,z)-\mu_{\pm}^{\infty}(x)e_1)-\alpha_{\pm}(x)\in L_{x}^{\infty}(\mathbb{R},L_z^2(\mathbb{R})),
\end{equation}
\begin{equation}\label{2.32}
z(\nu_{\pm}(x,z)-\nu_{\pm}^{\infty}(x)e_2)-\beta_{\pm}(x)\in L_{x}^{\infty}(\mathbb{R},L_z^2(\mathbb{R})),
\end{equation}
and the map
\begin{align}\label{2.45}
u\mapsto[z(\mu_{\pm}(x,z)-\mu_{\pm}^{\infty}(x)e_1)-\alpha_{\pm}(x), z(\nu_{\pm}(x,z)-\nu_{\pm}^{\infty}(x)e_2)-\beta_{\pm}(x)]
\end{align}
is Lipschitz continuous from $H^{2}(\mathbb{R})\cap H^{1,1}(\mathbb{R})$ to $L_{x}^{\infty}(\mathbb{R}, L_{z}^{2}(\mathbb{R}))$.
\end{lemma}
\begin{proof}
As usual, we still only give the proof of $\mu_{-}(x,z)$ and the rest is similar.
Before we start the formal proof, we define the operator $K$
\begin{align}\label{2.11}
Kf(x,z):=\int_{-\infty}^{x}\begin{bmatrix}1&0\\0&e^{2iz(x-y)}\end{bmatrix}Q_1(y)f(y,z)dy,
\end{align}
and give its corresponding properties. For every column vector $f(x,z)\in L_{x}^{\infty}(\mathbb{R}, L_{z}^{2}(\mathbb{R}))$, the following estimate can be obtained by mathematical induction
\begin{align*}
\|(K^{n}f)(x,z)\|_{L_{x}^{\infty}(\mathbb{R}, L_{z}^{2}(\mathbb{R}))}\leq \frac{\|Q_1\|_{L^1(\mathbb{R})}^{n}}{n!}\|f(x,z)\|_{L_{x}^{\infty}(\mathbb{R}, L_{z}^{2}(\mathbb{R}))},
\end{align*}
which implies that the operator $I-K$ is invertible in the space $L_{x}^{\infty}(\mathbb{R}, L_{z}^{2}(\mathbb{R}))$ and
\begin{equation}\label{2.38}
\left\|(I-K)^{-1}\right\|_{L_{x}^{\infty}(\mathbb{R}, L_{z}^{2}(\mathbb{R}))\rightarrow L_{x}^{\infty}(\mathbb{R}, L_{z}^{2}(\mathbb{R}))}\leq \sum_{n=0}^{\infty}\frac{\|Q_{1}\|_{L^1(\mathbb{R})}^{n}}{n!}=e^{\|Q_{1}\|_{L^1(\mathbb{R})}}.
\end{equation}
In fact, the operator $I-K$ is also invertible in space $L_{x}^{\infty}(\mathbb{R}^{-}, L_{z}^{2}(\mathbb{R}))$, again with
\begin{align}\label{2.38a}
\left\|(I-K)^{-1}\right\|_{L_{x}^{\infty}(\mathbb{R}^{-}, L_{z}^{2}(\mathbb{R}))\rightarrow L_{x}^{\infty}(\mathbb{R}^{-}, L_{z}^{2}(\mathbb{R}))}\leq \sum_{n=0}^{\infty}\frac{\|Q_{1}\|_{L^1(\mathbb{R})}^{n}}{n!}=e^{\|Q_{1}\|_{L^1(\mathbb{R})}}.
\end{align}

Thanks to the definition of the operator $K$, we can rewrite the integral equation \eqref{2.8} as
\begin{equation}\label{2.33}
(I-K)\mu_{-}(x,z)=e_{1}.
\end{equation}
Subtracting $(I-K)\mu_{-}^{\infty}e_{1}$ from both sides of the above equation, we have
\begin{equation}\label{2.34}
(I-K)(\mu_{-}-\mu_{-}^{\infty}e_1)=e_{1}-(I-K)\mu_{-}^{\infty}e_1\triangleq g(x,z)e_{2},
\end{equation}
where
\begin{equation*}
g(x,z)=\int_{-\infty}^{x}e^{2iz(x-y)}w(y)dy,\ \ w(x)=\partial_{x}[v(x)e^{-\frac{1}{2i}\int_{-\infty}^xuvdy}].
\end{equation*}
If $u\in H^{1,1}(\mathbb{R})$, it is easy to obtain $w(x)\in L^{2,1}(\mathbb{R})$. By Proposition 1 in reference \cite{Pel} it follows that $g(x,z)\in L_{x}^{\infty}(\mathbb{R}, L_{z}^{2}(\mathbb{R}))$.
Then, we can obtain $\mu_{-}(x,z)-\mu_{-}^{\infty}(x)e_{1}\in L_{x}^{\infty}(\mathbb{R}, L_{z}^{2}(\mathbb{R}))$ by using \eqref{2.38} and \eqref{2.34}.

We take derivative of the equation \eqref{2.33} in $z$ and obtain
\begin{equation}\label{2.40}
(I-K)\eta(x,z)=g_{1}(x,z)e_{1}+g_{2}(x,z)e_{2}+g_{3}(x,z)e_{2},
\end{equation}
where
\begin{align*}
&\eta(x,z)=[\partial_{z}\mu_{-}^{(1)},\partial_{z}\mu_{-}^{(2)}-2ix\mu_{-}^{(2)}]^T,\\
&g_{1}(x,z)=\int_{-\infty}^{x}yu(y)\mu_{-}^{(2)}(y,z)dy,\\
&g_{2}(x,z)=-\int_{-\infty}^{x}ye^{2iz(x-y)}[2iv_{y}(y)-v^2(y)u(y)][\mu_{-}^{(1)}(y,z)-\mu_{-}^{\infty}(y)]dy,\\
&g_{3}(x,z)=-\int_{-\infty}^{x}ye^{2iz(x-y)}[2iv_{y}(y)-v^2(y)u(y)]\mu_{-}^{\infty}(y)dy.
\end{align*}
Combining \eqref{2.38a}, \eqref{2.34} and  Proposition 1 in \cite{Pel}, we can get the following estimates
\begin{align*}
&\sup_{x\in\mathbb{R}^{-}}\|g_{1}(x,z)e_{1}+g_{2}(x,z)e_{2}+g_{3}(x,z)e_{2}\|_{L_{z}^{2}(\mathbb{R})}\\
\leq& C\sup_{x\in\mathbb{R}^{-}}\|\langle x \rangle[\mu_{-}(x,z)-\mu_{-}^{\infty}(x)e_1]\|_{L_{z}^{2}(\mathbb{R})}+\tilde{C}\\
\leq& C\sup_{x\in\mathbb{R}^{-}}\|\langle x\rangle g(x,z)e_{2}\|_{L^{2}_{z}(\mathbb{R})}+\tilde{C} \\
\leq& C\|w\|_{L^{2,1}(\mathbb{R}^{-})}+\tilde{C} ,
\end{align*}
where $C$ and $\tilde{C} $ are two positive constants depend on $\|u\|_{H^{1,1}(\mathbb{R})}$. Using the property \eqref{2.38a} again gives $\eta(x,z)\in L_{x}^{\infty}(\mathbb{R}^{-}, L_{z}^{2}(\mathbb{R}))$, which implies $\partial_{z}[\mu_{-}(x,z)-\mu_{-}^{\infty}(x)e_{1}]\in L_{x}^{\infty}(\mathbb{R}^{-}, L_{z}^{2}(\mathbb{R}))$. Therefore, we complete the proof of $\mu_{-}$ in \eqref{2.30}.

Next, we prove that the map \eqref{2.43} is Lipschitz continuous. Suppose that $u, \tilde{u} \in H^{1,1}(\mathbb{R})$ satisfy $\|u\|_{H^{1,1}(\mathbb{R})}, \|\tilde{u}\|_{H^{1,1}(\mathbb{R})}\leq \gamma$ for some $\gamma>0$. Denote the corresponding Jost functions by $\mu_{-}(x,z)$ and $\tilde{\mu}_{-}(x,z)$ respectively. And we can define $\tilde{K}, \tilde{g}, \tilde{w}$ corresponding to $\tilde{u}$ similarly to the equations \eqref{2.11} and \eqref{2.34}, then
\begin{align*}
&(\mu_{-}-\mu_{-}^{\infty}e_{1})-(\tilde{\mu}_{-}-\tilde{\mu}_{-}^{\infty}e_1)\\
=&(I-K)^{-1}ge_{2}-(I-\tilde{K})^{-1}\tilde{g}e_{2}\\
=&(I-K)^{-1}(g-\tilde{g})e_{2}+(I-K)^{-1}(K-\tilde{K})(I-\tilde{K})^{-1}\tilde{g}e_{2}.
\end{align*}
According to \eqref{2.38} and Proposition 1 in  \cite{Pel}, we can know that
\begin{align*}
&\sup_{x\in\mathbb{R}}\|(I-K)^{-1}(g-\tilde{g})e_{2}\|_{L_{z}^{2}}\\
\leq& c_{1}(\gamma)\sup_{x\in\mathbb{R}}\|g-\tilde{g}\|_{L_{z}^{2}}\\
\leq &c_{1}(\gamma)\|w-\tilde{w}\|_{L^{2}}\\
\leq & c_{1}(\gamma)\|\mu_{-}^{\infty}-\tilde{\mu}_{-}^{\infty}\|_{L^{\infty}}+c_{2}(\gamma)\|u-\tilde{u}\|_{H^{1,1}}\\
\leq & c_{1}(\gamma)\|\int_{-\infty}^{x}(|u(y)|^{2}-|\tilde{u}(y)|^{2})dy\|_{L^{\infty}}+c_{2}(\gamma)\|u-\tilde{u}\|_{H^{1,1}}\\
\leq &c_{3}(\gamma)\|u-\tilde{u}\|_{H^{1,1}},
\end{align*}
and
\begin{align*}
&\sup_{x\in\mathbb{R}}\|(I-K)^{-1}(K-\tilde{K})(I-\tilde{K})^{-1}\tilde{g}e_{2}\|_{L_{z}^{2}}\\
\leq&c_{4}(\gamma) \sup_{x\in\mathbb{R}}\|(K-\tilde{K})(I-\tilde{K})^{-1}\tilde{g}e_{2}\|_{L_{z}^{2}}\\
\leq&c_{4}(\gamma) \sup_{x\in\mathbb{R}}\|(I-\tilde{K})^{-1}\tilde{g}e_{2}\|_{L_{z}^{2}}\|u-\tilde{u}\|_{H^{1,1}}\\
\leq&c_{4}(\gamma) \sup_{x\in\mathbb{R}}\|\tilde{g}e_{2}\|_{L_{z}^{2}}\|u-\tilde{u}\|_{H^{1,1}}\\
\leq &c_{4}(\gamma) \|\tilde{w}\|_{L^{2}}\|u-\tilde{u}\|_{H^{1,1}}\\
\leq &c_{4}(\gamma) \|u-\tilde{u}\|_{H^{1,1}},
\end{align*}
where $c_{3}(\gamma)$ and $c_{4}(\gamma)$ are two positive $\gamma$-dependent constants.
Thus, we have
\begin{align*}
\sup_{x\in\mathbb{R}}\|(\mu_{-}-\mu_{-}^{\infty}e_{1})-(\tilde{\mu}_{-}-\tilde{\mu}_{-}^{\infty}e_1)\|_{L_{z}^{2}}\leq c(\gamma) \|u-\tilde{u}\|_{H^{1,1}},
\end{align*}
where $c(\gamma)$ is a positive $\gamma$-dependent constant. A similar analysis of \eqref{2.40} using \eqref{2.38a} and Proposition 1 in \cite{Pel} shows that there exists $c(\gamma)$ such that
\begin{align*}
\sup_{x\in\mathbb{R}^{-}}\|\partial_{z}(\mu_{-}-\mu_{-}^{\infty}e_{1})-\partial_{z}(\tilde{\mu}_{-}-\tilde{\mu}_{-}^{\infty}e_1)\|_{L_{z}^{2}}\leq c(\gamma) \|u-\tilde{u}\|_{H^{1,1}},
\end{align*}
Therefore, we prove the Lipschitz continuity of \eqref{2.43}.

In order to prove \eqref{2.31}, we can imitate \eqref{2.34} and define $\hat{g}(x,z)$ as follows
\begin{align*}
&(I-K)[z(\mu_{-}-\mu_{-}^{\infty}e_{1})-\alpha_{-}]\\
=&zge_{2}-(I-K)\alpha_{-}\\
\triangleq& \hat{g}(x,z)e_{2}.
\end{align*}
By applying the similar analysis to the above equation, we can prove that the result \eqref{2.31} and the map \eqref{2.45} is Lipschitz continuous.
\end{proof}

\subsection{Regularity of the scattering data}
In order to build the scattering data, we need to consider the Jost solution of the original spectral problem \eqref{1.9}. Based on the transformation \eqref{2.3} and the definition of $\mu(x,z), \nu(x,z)$ \eqref{2.6}, we can define the normalized Jost functions for the spectral problem \eqref{1.9} as follows
\begin{equation}\label{2.51}
\varphi_{\pm}(x,k)=T_1^{-1}(x,k)\mu_{\pm}(x,z),\ \ \psi_{\pm}(x,k)=T_2^{-1}(x,k)\nu_{\pm}(x,z),\ \ k\in\mathbb{C}\setminus\{0\},
\end{equation}
then, the two Jost functions each satisfy the following Volterra's integral equation
\begin{equation}\label{2.53}
\varphi_{\pm}(x,k)=e_1+k\int_{\pm\infty}^{x}\begin{pmatrix}1&0\\0&e^{2ik^2(x-y)}\end{pmatrix}Q(y)\varphi_{\pm}(y,z)dy,
\end{equation}
\begin{equation}\label{2.54}
\psi_{\pm}(x,k)=e_2+k\int_{\pm\infty}^{x}\begin{pmatrix}e^{2ik^2(x-y)}&0\\0&1\end{pmatrix}Q(y)\psi_{\pm}(y,z)dy.
\end{equation}
When $k=0$, it is obvious that $\varphi_{\pm}(x,0)=e_1$, $\psi_{\pm}(x,0)=e_2$ and $\mu_{\pm}(x,0)=[1,v(x)]^{T}$, $\nu_{\pm}(x,0)=[-u(x),1]^{T}$.

We note that $\varphi_{\pm}=[\varphi_{\pm}^{(1)},\varphi_{\pm}^{(2)}]^{T}$ and $\psi_{\pm}=[\psi_{\pm}^{(1)},\psi_{\pm}^{(2)}]^{T}$.
According to Lemma \ref{l1}, \eqref{2.7} and the definition \eqref{2.51}, it is easy to obtain the following corollary.
\begin{corollary}\label{c2}
If $u\in H^{1,1}(\mathbb{R})$, then $\varphi_{\pm}(x,k)$ and $\psi_{\pm}(x,k)$ admit the following properties
\begin{enumerate}[label=(\roman*)]
  \item $\forall \ k^2\in\mathbb{R}\backslash\{0\}$, integral equations \eqref{2.53} and \eqref{2.54}  there exists unique solutions $ \varphi_{\pm}(x,k)\in L_{x}^{\infty}(\mathbb{R})$ and $ \psi_{\pm}(x,k)\in L_{x}^{\infty}(\mathbb{R})$ respectively.
  \item $\forall x\in\mathbb{R}$, $\varphi_{-}(x,k)$ and $\psi_{+}(x,k)$ are analytic in the first and third quadrant of the $k$ plane, $\varphi_{+}(x,k)$ and $\psi_{-}(x,k)$  are analytic in the second and fourth quadrant of the $k$ plane.
  \item $\varphi_{\pm}(x,k)\rightarrow e_1, \ \psi_{\pm}(x,k)\rightarrow e_2, \ \ x\rightarrow \pm\infty$.
  \item $\varphi_{\pm}^{(1)}(x,k), \ \psi_{\pm}^{(2)}(x,k)$ are even in $k$, $\varphi_{\pm}^{(2)}(x,k), \ \psi_{\pm}^{(1)}(x,k)$ are odd in $k$.
  \item $\varphi_-(x,k)=\begin{bmatrix}0&1\\ \pm1&0 \end{bmatrix} \overline{\psi_+(-x,\pm i\overline{k})},\
      \ \psi_-(x,k)=\begin{bmatrix}0&\pm1\\1 &0 \end{bmatrix}\overline{\varphi_+(-x,\pm i\overline{k})}$.
\end{enumerate}
\end{corollary}

By the theory of ODE, we can define the following scattering matrix associated with the spectral problem \eqref{1.9}
\begin{align}\label{2.59}
\begin{aligned}
&\begin{bmatrix}\varphi_-(x,k)e^{-ik^2x}&\psi_-(x,k)e^{ik^2x}\end{bmatrix}\\
=&\begin{bmatrix}\varphi_+(x,k)e^{-ik^2x}&\psi_+(x,k)e^{ik^2x}\end{bmatrix}\begin{bmatrix}a(k)&c(k)\\b(k)&d(k)\end{bmatrix},
\end{aligned}
\end{align}
then, the scattering data $a(k), b(k), c(k), d(k)$ are related to the Wronskian of the system via the relations below
\begin{align}
&a(k)=W(\varphi_-(x,k)e^{-ik^2x},\psi_+(x,k)e^{ik^2x}),\label{2.60}\\
&b(k)=W(\varphi_+(x,k)e^{-ik^2x},\varphi_-(x,k)e^{-ik^2x}),\label{2.61}\\
&c(k)=W(\psi_-(x,k)e^{ik^2x},\psi_+(x,k)e^{ik^2x}),\label{2.61c}\\
&d(k)=W(\varphi_+(x,k)e^{-ik^2x},\psi_-(x,k)e^{ik^2x})\label{2.60d}.
\end{align}

The scattering data $a(k), b(k)$ and $d(k)$ can also be expressed in the following integral form
\begin{align}
&a(k)=1+k\int_{-\infty}^{+\infty}u(y)\varphi_-^{(2)}(y,k)dy=1-k\int_{-\infty}^{+\infty}v(y)\psi_+^{(1)}(y,k)dy,\label{a}\\
&b(k)=k\int_{-\infty}^{+\infty}v(y)\varphi_-^{(1)}(y,k)e^{-2ik^2y}dy=k\int_{-\infty}^{+\infty}u(y)\varphi_+^{(1)}(y,k)e^{-2ik^2y}dy,\label{b}\\
&d(k)=1-k\int_{-\infty}^{+\infty}u(y)\varphi_+^{(2)}(y,k)dy=1+k\int_{-\infty}^{+\infty}v(y)\psi_-^{(1)}(y,k)dy.\label{d}
\end{align}
Further, we can obtain the following properties of the scattering data from  Lemma \ref{l2}, \ref{l3}, Corollary \ref{c2} and the relation \eqref{2.51}.
\begin{corollary}\label{c2+}
If $u\in H^{1,1}(\mathbb{R})$, then the scattering data $a(k), b(k), c(k), d(k)$ admit the following properties
\begin{enumerate}[label=(\roman*)]
  \item $a(k)$ and $d(k)$ extend analytically to the first and third quadrant of the $k$ plane and the first and third quadrant of the $k$ plane, respectively.
  \item $a(k)\rightarrow e^{-\frac{1}{2i}\int_{-\infty}^{+\infty}uvdy}\triangleq a_{\infty}, \ $ $d(k)\rightarrow e^{\frac{1}{2i}\int_{-\infty}^{+\infty}uvdy}\triangleq d_{\infty},\ \ |k|\rightarrow\infty.$
  \item $a(k),d(k)$ are even in $k$, and $b(k),c(k)$ are odd in $k$.
  \item $a(k)=\overline{a(\pm i\overline{k})}, \ d(k)=\overline{d(\pm i\overline{k})}, \ c(k)=\mp\overline{b(\pm i\overline{k})},\ \ k\in\mathbb{R}\cup i\mathbb{R}.$
  \item $a(k)d(k)-b(k)c(k)=1,\ \ k\in\mathbb{R}\cup i\mathbb{R}.$
\end{enumerate}
\end{corollary}

Note that property (iv) \cite{Abl} in Corollary \ref{c2+}, the scattering data $a(k)$ and $d(k)$ are not related, which is a distinctive feature of the nDNLS equation compared to its conventional counterpart. This feature makes our subsequent analysis more difficult, and we will subsequently give specific ways to overcome this difficulty.

\begin{lemma}\label{l4}
If $u\in H^{1,1}(\mathbb{R})$, then we have
\begin{equation}\label{2.65}
a(k)-a_{\infty},\  d(k)-d_{\infty},\ kb(k),\ k^{-1}b(k),\ kc(k),\ k^{-1}c(k) \in H_{z}^{1}(\mathbb{R}),
\end{equation}
and the map
\begin{equation}\label{2.75}
u\mapsto[a(k)-a_{\infty},  d(k)-d_{\infty}, kb(k), k^{-1}b(k), kc(k), k^{-1}c(k)]
\end{equation}
is Lipschitz continuous from $H^{1,1}(\mathbb{R})$ to $H_{z}^{1}(\mathbb{R})$.

Moreover, if $u\in H^{2}(\mathbb{R})\cap H^{1,1}(\mathbb{R})$, then we have
\begin{equation}\label{2.66}
kb(k),\ k^{-1}b(k),\ kc(k), \ k^{-1}c(k) \in L_{z}^{2,1}(\mathbb{R}).
\end{equation}
and the map
\begin{equation}\label{2.76}
u\mapsto[kb(k),k^{-1}b(k), kc(k), k^{-1}c(k)]
\end{equation}
is Lipschitz continuous from $H^{2}(\mathbb{R})\cap H^{1,1}(\mathbb{R})$ to $L_{z}^{2,1}(\mathbb{R})$.
\end{lemma}
\begin{proof}
The proof mainly relies on Lemma \ref{l3} and the relation \eqref{2.51}, where \eqref{2.51} implies that
\begin{align}\label{2.51a}
\left\{\begin{aligned}
&\varphi_{\pm}^{(1)}(x,k)=\mu_{\pm}^{(1)}(x,z),\\
&\varphi_{\pm}^{(2)}(x,k)=-\frac{1}{2ik}v(x)\mu_{\pm}^{(1)}(x,z)+\frac{1}{2ik}\mu_{\pm}^{(2)}(x,z),
\end{aligned}
\right.
\end{align}
and
\begin{align}\label{2.51b}
\left\{\begin{aligned}
&\psi_{\pm}^{(1)}(x,k)=\frac{1}{2ik}\nu_{\pm}^{(1)}(x,z)+\frac{1}{2ik}u(x)\nu_{\pm}^{(2)}(x,z),\\
&\psi_{\pm}^{(2)}(x,k)=\nu_{\pm}^{(2)}(x,z).
\end{aligned}
\right.
\end{align}

By the Wronskian determinant \eqref{2.60}, property (ii) in the corollary \ref{c2+} and \eqref{2.51a}, we have
\begin{align}\label{2.69}
\begin{aligned}
a(k)-a_{\infty}=&\det\begin{bmatrix}\mu_{-}^{(1)}(x,z)&\psi_{+}^{(1)}(x,k)\\ \varphi_{-}^{(2)}(x,k)&\nu_{+}^{(2)}(x,z)\end{bmatrix}-\det\begin{bmatrix}\mu_{-}^{\infty}(x)&0\\0&\nu_{+}^{\infty}(x)\end{bmatrix}\\
=&[\mu_{-}^{(1)}(0,z)-\mu_{-}^{\infty}(0)][\nu_{+}^{(2)}(0,z)-\nu_{+}^{\infty}(0)]-\varphi_{-}^{(2)}(0,k)\psi_{+}^{(1)}(0,k)\\
&+\mu_{-}^{\infty}(0)[\nu_{+}^{(2)}(0,z)-\nu_{+}^{\infty}(0)]+\nu_{+}^{\infty}(0)[\mu_{-}^{(1)}(x,z)-\mu_{-}^{\infty}(0)].
\end{aligned}
\end{align}
We rewrite the term $\varphi_{-}^{(2)}(0,k)\psi_{+}^{(1)}(0,k)$ in \eqref{2.69} as
\begin{align}\label{2.69+}
\begin{aligned}
\varphi_{-}^{(2)}(0,k)\psi_{+}^{(1)}(0,k)=&\frac{1}{2i}k^{-1}\varphi_{-}^{(2)}(0,k)u(0)\nu_{+}^{\infty}(0)\\
&+\frac{1}{2i}k^{-1}\varphi_{-}^{(2)}(0,k)[2ik\psi_{+}^{(1)}(0,k)-u(0)\nu_{+}^{\infty}(0)],
\end{aligned}
\end{align}
where
\begin{align*}
2ik\psi_{+}^{(1)}(0,k)-u(0)\nu_{+}^{\infty}(0)=\nu_{+}^{(1)}(0,z)+u(0)[\nu_{+}^{(2)}(0,z)-\nu_{+}^{\infty}(0,z)]\in H_{z}^{1}(\mathbb{R}),
\end{align*}
according to \eqref{2.51b} and \eqref{2.30}.

Below we prove that $k^{-1}\varphi_{\pm}^{(2)}(0,k)\in H_{z}^{1}(\mathbb{R})$. From the integral expression \eqref{2.53} we can directly obtain
\begin{align*}
k^{-1}\varphi_{\pm}^{(2)}(0,k)=\int_{\pm\infty}^{0}v(y)e^{-2izy}\mu_{\pm}^{(1)}(y,z)dy,
\end{align*}
therefore,
\begin{align*}
&\|k^{-1}\varphi_{\pm}^{(2)}(0,k)\|_{L_{z}^{2}}\\ \leq&\|\int_{\pm\infty}^{0}v(y)e^{-2izy}[\mu_{\pm}^{(1)}(y,z)-\mu_{\pm}^{\infty}(y)]dy\|_{L_{z}^{2}}+\|\int_{\pm\infty}^{0}v(y)e^{-2izy}\mu_{\pm}^{\infty}(y)dy\|_{L_{z}^{2}}\\
\leq&\int_{\pm\infty}^{0}|v(y)|\|\mu_{\pm}^{(1)}(y,z)-\mu_{\pm}^{\infty}(y)\|_{L^{2}_{z}}dy+\|v\|_{L^{2}}\\
\leq&\sup_{x\in\mathbb{R}}\|\mu_{\pm}^{(1)}(x,z)-\mu_{\pm}^{\infty}(x)\|_{L^{2}_{z}}\|v\|_{L^{1}}+\|v\|_{L^{2}}.
\end{align*}
And $\partial_{z}(k^{-1}\varphi_{\pm}^{(2)}(0,k))\in L_{z}^{2}(\mathbb{R})$ can be proved similarly, therefore we have $ k^{-1}\varphi_{\pm}^{(2)}(0,k)\in H_{z}^{1}(\mathbb{R})$.

By \eqref{2.69+} and  the Banach algebra property of $H_{z}^{1}(\mathbb{R})$, we have
\begin{equation*}
\varphi_{-}^{(2)}(0,k) \psi_{+}^{(1)}(0,k)\in H_{z}^{1}(\mathbb{R}).
\end{equation*}
Further combining \eqref{2.69} and \eqref{2.30} we can get $a(k)-a_{\infty}\in H_{z}^{1}(\mathbb{R})$. By using the same process as above, we can obtain $d(k)-d_{\infty}\in H_{z}^{1}(\mathbb{R})$.

From the Wronskian determinant representation \eqref{2.61} and \eqref{2.51a} we have
\begin{align}\label{2.70}
\begin{aligned}
2ikb(k)=&\det\begin{bmatrix}\varphi_{+}^{(1)}(0,k)&\varphi_{-}^{(1)}(0,k)\\2ik\varphi_{+}^{(2)}(0,k)&2ik\varphi_{-}^{(2)}(0,k)\end{bmatrix}\\
=&\det\begin{bmatrix}\mu_{+}^{(1)}(0,z)&\mu_{-}^{(1)}(0,z)\\ \mu_{+}^{(2)}(0,z)&\mu_{-}^{(2)}(0,z)\end{bmatrix},
\end{aligned}
\end{align}
so $kb(k)\in H_{z}^{1}(\mathbb{R})$ is easily obtained according to \eqref{2.30}. The same is true for $kc(k)$.

Again, thanks to the Wronskian determinant of the scattering data, using \eqref{2.61c} we get the following equation
\begin{align*}
k^{-1}b(k)=\det\begin{bmatrix}\mu_{+}^{(1)}(0,z)&\mu_{-}^{(1)}(0,z)\\ k^{-1}\varphi_{+}^{(2)}(0,k)&k^{-1}\varphi_{-}^{(2)}(0,k)\end{bmatrix},
\end{align*}
thus, $k^{-1}b(k)\in H_{z}^{1}(\mathbb{R})$. The regularity $k^{-1}\psi_{\pm}^{(1)}(0,k) \in  H_{z}^{1}(\mathbb{R})$ can be obtained similarly, as proven by $k^{-1}\varphi_{\pm}^{(2)}(0,k)\in H_{z}^{1}(\mathbb{R})$, that therefore $k^{-1}c(k)\in H_{z}^{1}(\mathbb{R})$. This completes the proof of \eqref{2.65}.

The conclusion \eqref{2.66} can be obtained from \eqref{2.31}, \eqref{2.32} and the  Wronskian determinants \eqref{2.61}, \eqref{2.61c}. The Lipschitz continuity \eqref{2.75} and \eqref{2.76} follows from the Lipschitz continuity \eqref{2.43} and \eqref{2.45}.
\end{proof}

\section{Construction of the RH problems}
\subsection{The relationship between RH problems}

We define the reflection coefficients by
\begin{align}\label{3.1}
r_1(k):=\frac{b(k)}{a(k)},\ \ r_2(k):=\frac{c(k)}{d(k)},\ \ k\in \mathbb{R}\cup i\mathbb{R}.
\end{align}

From the relation \eqref{2.59} and the corollary \ref{c2} , \ref{c2+}, we can define a sectionally analytical matrix
\begin{align}
\Psi(x,k):=\left\{
\begin{aligned}
&[\frac{\varphi_{-}(x,k)}{a(k)},\psi_{+}(x,k)],\ \ \text{Im}k^2>0,\\
&[\varphi_+(x,k), \frac{\psi_-(x,k)}{d(k)}],\ \ \text{Im}k^2<0,
\end{aligned}
\right.
\end{align}
then $\Psi(x,k)$ solves the following RH problem.
\begin{rhp}\label{RH1}
Find a matrix-valued function $\Psi(x,k)$ that satisfies the following conditions
\begin{itemize}
  \item Analyticity: $\Psi(x,k)$  is analytical in $\mathbb{C}\setminus \{\mathbb{R}\cup i\mathbb{R}\}$.
  \item Jump condition:  $\Psi(x,k)$ has continuous boundary values $\Psi_{\pm}(x,k)$ on $\mathbb{R}\cup i\mathbb{R}$ and
  \begin{align}\label{3.6}
  \Psi_+(x,k)-\Psi_-(x,k)=\Psi_-(x,k)S(x,k), \ \ k\in \mathbb{R}\cup i\mathbb{R},
  \end{align}
  where
  \begin{align}\label{3.7}
  S(x,k)=\begin{bmatrix}-r_1(k)r_2(k)&-r_2(k)e^{-2ik^2x}\\r_1(k)e^{2ik^2x}&0\end{bmatrix}.
  \end{align}
  \item Asymptotic condition:
  \begin{align}\label{3.9}
  \Psi(x,k)\rightarrow[e^{\frac{1}{2i}\int^{+\infty}_{x}uvdy}e_1,e^{-\frac{1}{2i}\int^{+\infty}_{x}uvdy}e_2]\triangleq\Psi_{\infty}(x),\ \ |k|\rightarrow\infty.
  \end{align}
\end{itemize}
\end{rhp}

Notice that the reconstruction formulas \eqref{2.17} and \eqref{2.18} are built in the $z$-plane, so we next need to consider the RH problem in the $z$-plane.
First, the reflection coefficients in the $z$-plane are established as follows
\begin{align}\label{3.17}
r_{-}(z):=2ikr_1(k)=\frac{2ikb(k)}{a(k)},\ \ r_{+}(z):=-\frac{r_2(k)}{2ik}=-\frac{c(k)}{2ikd(k)},\ \ z\in\mathbb{R}.
\end{align}

Starting from \eqref{2.59} and using the relation \eqref{2.51}, a sectionally analytical matrix function can be defined by
\begin{align}\label{3.22}
\Gamma(x,z):=\left\{
\begin{aligned}
&[\frac{\mu_-(x,z)}{a(z)},\gamma_+(x,z)],\ \ \text{Im}z>0,\\
&[\mu_+(x,z),\frac{\gamma_{-}(x,z)}{d(z)}],\ \ \text{Im}z<0,
\end{aligned}
\right.
\end{align}
where
\begin{equation}\label{3.13}
\begin{aligned}
\gamma_{\pm}(x,z)=&\frac{1}{2ik}T_1(x,k)T_2^{-1}(x,k)\nu_{\pm}(x,z)\\
=&-\frac{1}{4z}\begin{bmatrix}1&u(x)\\v(x)&u(x)v(x)-4z\end{bmatrix}\nu_{\pm}(x,z),
\end{aligned}
\end{equation}
and $a(z)=a(k), d(z)=d(k)$, this notation is reasonable because both $a(k)$ and $d(k)$ are even functions with respect to $k$. $\Gamma(x,z)$ satisfies the following RH problem in $z$-plane.
\begin{rhp}\label{RH2}
Find a matrix-valued function $\Gamma(x,z)$ that satisfies the following conditions
\begin{itemize}
  \item Analyticity: $\Gamma(x,z)$  is analytical in $\mathbb{C}\setminus \mathbb{R}$.
  \item Jump condition:  $\Gamma(x,z)$ has continuous boundary values $\Gamma_{\pm}(x,z)$ on $\mathbb{R}$ and
  \begin{equation}
  \Gamma_{+}(x,z)-\Gamma_{-}(x,z)=\Gamma_{-}(x,z)R(x,z),\ \ z\in\mathbb{R},
  \end{equation}
  where
  \begin{align}\label{3.23}
  R(x,z)=\begin{bmatrix}r_{+}(z)r_{-}(z)&r_{+}(z)e^{-2izx}\\r_{-}(z)e^{2izx}&0\end{bmatrix}.
  \end{align}
  \item Asymptotic condition:
  \begin{equation}\label{3.24}
  \Gamma(x,z)\rightarrow\Psi_{\infty}(x),\ \ |z|\rightarrow\infty.
  \end{equation}
\end{itemize}
\end{rhp}

We can further normalize the above RH problem \ref{RH2} to a RH problem with the identity matrix as the boundary condition by defining the following matrix
\begin{equation}\label{3.25}
M(x,z):=[\Psi_{\infty}(x)]^{-1}\Gamma(x,z),
\end{equation}
which then satisfy the following RH problem.
\begin{rhp}\label{RH3}
Find a matrix-valued function $M(x,z)$ that satisfies the following conditions
\begin{itemize}
  \item Analyticity: $M(x,z)$  is analytical in $\mathbb{C}\setminus \mathbb{R}$.
  \item Jump condition:  $M(x,z)$ has continuous boundary values $M_{\pm}(x,z)$ on $\mathbb{R}$ and
  \begin{equation}\label{3.26}
  M_{+}(x,z)-M_{-}(x,z)=M_{-}(x,z)R(x,z),\ \ z\in\mathbb{R}.
  \end{equation}
  \item Asymptotic condition:
  \begin{align}
  M(x,z)\rightarrow I,\ \ |z|\rightarrow\infty.
  \end{align}
\end{itemize}
\end{rhp}

Lemma \ref{l2} gives the connection formulas between the potential $u(x)$ and the Jost functions in the $z$-plane, and this connection can be further used to obtain properties of the potential $u(x)$ by studying the RH problem $M(x,z)$ in the $z$-plane. Therefore, the solvability of the RH problem \ref{RH3} plays an important role in the subsequent deduction.

At the beginning, we considered starting directly from RH Problem \ref{RH3} by adding small norm restriction on the reflection coefficients $r_{\pm}(z)$ in order to be able to use Zhou vanishing lemma \cite{Zho} for the jump matrix $R(x,z)$. However, we found that the reflection coefficients $r_{\pm}(z)$ in the $z$-plane did not have good enough properties to allow us to draw further conclusions to prove Theorem \ref{t1}. Therefore, in the following we will further transform RH Problem \ref{RH3} by creating two new RH problems related to the matrix $S(x,k)$ in \eqref{3.7} and using vanishing lemma for $S(x,k)$.

Observing the expressions \eqref{3.7} and \eqref{3.23} for the matrices $S(x,k)$ and $R(x,z)$ respectively, we find that $S(x,k)$ and $R(x,z)$ are related as follows
\begin{equation}\label{3.27-}
R(x,z)\rho_j(k)=\rho_j(k)S(x,k),\ \ z\in\mathbb{R},\ \ k\in\mathbb{R}\cup i\mathbb{R},\ \ j=1,2,
\end{equation}
where
\begin{align}\label{3.27}
\rho_1(k)=\begin{bmatrix}1&0\\0&2ik\end{bmatrix},\ \ \rho_2(k)=\begin{bmatrix}\frac{1}{2ik}&0\\0&1\end{bmatrix},\ \ k\in\mathbb{C}\setminus\{0\}.
\end{align}

We establish two new matrices
\begin{equation}\label{3.29}
N_{j}(x,k):=M(x,z)\rho_{j}(k)-\rho_{j}(k),\ \ j=1,2.
\end{equation}
It is easy to derive that $N_{j}(x,k)$ satisfies the new RH problems:
\begin{rhp}\label{RH4}
Find the matrix-valued function $N_{j}(x,k),\ j=1,2$ that satisfies the following conditions
\begin{itemize}
  \item Analyticity: $N_{j}(x,k)$  is analytical in $\mathbb{C}\setminus \{\mathbb{R}\cup i\mathbb{R}\}$.
  \item Jump condition:  $N_{j}(x,k)$ has continuous boundary values $N_{j,\pm}(x,k)$ on $\mathbb{R}\cup i\mathbb{R}$ and
  \begin{align}\label{3.28}
  N_{j,+}(x,k)-N_{j,-}(x,k)=N_{j,-}(x,k)S(x,k)+F_{j}(x,k),\ \ k\in\mathbb{R}\cup i\mathbb{R},
  \end{align}
  where $F_{j}(x,k):=\rho_{j}(k)S(x,k),\ j=1,2.$
  \item Asymptotic condition:
  \begin{equation}
  N_{j,\pm}(x,k)\rightarrow 0,\ \ |k|\rightarrow\infty.
  \end{equation}
\end{itemize}
\end{rhp}

\subsection{Reflection coefficients and Lipschitz continuity}
We define the Hermitian part of $I+S(x,k)$ by
\begin{align*}
I+S_{H}(x,k):=&I+\frac{1}{2}[S(x,k)+S^{H}(x,k)]\\
=&\begin{bmatrix}1-\text{Re}[r_1(k)r_2(k)]&\frac{1}{2}[\bar{r}_1(k)-r_2(k)]e^{-2ik^2x}\\ \frac{1}{2}[r_1(k)-\bar{r}_2(k)]e^{2ik^2x}&1\end{bmatrix},
\end{align*}
where the superscript $H$ means Hermite conjugate.
In order to ensure the solvability of the RH problems, it is necessary to require that $I+S_{H}(x,k)$ be positive definite matrix, i.e. with the following restrictions on the reflection coefficients $r_{1}(k)$ and $r_{2}(k)$
\begin{align}\label{pd}
1-\text{Re}[r_1(k)r_2(k)]>0,\ \ 1-\frac{1}{4}|r_{1}(k)+\bar{r}_2(k)|^2>0.
\end{align}
Next, we present the following Proposition which gives a sufficient condition for \eqref{pd}.
\begin{proposition}\label{psmall}
If $u\in H^{1,1}(\mathbb{R})$ satisfies the
small norm restriction
\begin{align}\label{small+}
\|Q_{1}(u)\|_{L^{1}(\mathbb{R})}\leq0.295,
\end{align}
then $|r_{j}(k)|<1, j=1,2$, for every $k\in \mathbb{R}\cup i\mathbb{R}$.
\end{proposition}
\begin{proof}
Let us define the operator $K_{1}$ by
\begin{align}\label{k1}
K_{1}f(x,z):=e_1+\int_{-\infty}^{x}\begin{pmatrix}1&0\\0&e^{2iz(x-y)}\end{pmatrix}Q_1(u(y))f(y,z)dy.
\end{align}
For every $\text{Im}z\geq0$, we assume that $f(x,z)=[f_{1}(x,z),f_{1}(x,z)]^{T}$, $g(x,z)=[g_{1}(x,z),g_{1}(x,z)]^{T} \in L_{x}^{\infty}(\mathbb{R})$, then
\begin{align*}
K_{1}f-K_{1}g=\begin{bmatrix}\frac{1}{2i}\int_{-\infty}^{x}-uv(f_{1}-g_{1})+u(f_{2}-g_{2})dy\\
\frac{1}{2i}\int_{-\infty}^{x}[(2iv_{y}-v^2u)(f_{1}-g_{1})+uv(f_{2}-g_{2})]e^{2iz(x-y)}dy
\end{bmatrix},
\end{align*}
thus, for every $\text{Im}z\geq0$,
\begin{align*}
\|K_{1}f-K_{1}g\|_{L^{\infty}_x}\leq&\frac{1}{2}(2\|\partial_{x}v\|_{L^{1}}+\|uv^{2}\|_{L^{1}}+2\|uv\|_{L^{1}}+\|u\|_{L^{1}})\|f-g\|_{L^{\infty}_x}\\
=& \|Q_{1}(u)\|_{L^{1}}\|f-g\|_{L^{\infty}_x}.
\end{align*}

We assume that
\begin{equation}
\|Q_{1}(u)\|_{L^{1}}< c,
\end{equation}
where $c$ is a positive constant. According to the definition of the operator $K_{1}$ in \eqref{k1} and  integral  expression \eqref{2.8}, it is easy to obtain
$\mu_{-}(x,z)=K_{1}\mu_{-}(x,z)$.
We set $f(x,z)=\mu_{-}(x,z),\ g(x,z)=[0,0]^{T}$, then for every $\text{Im}z\geq0$, we have
\begin{equation*}
\|\mu_{-}(x,z)-e_1\|_{L^{\infty}_x}=\|K_{1}\mu_{-}(x,z)-e_1\|_{L^{\infty}_x}< c\|\mu_{-}(x,z)\|_{L_x^{\infty}}.
\end{equation*}
Moreover,
\begin{equation}\label{gj}
\|\mu_{-}(x,z)\|_{L_x^{\infty}}<\frac{1}{1-c},\ \ \|\mu_{-}(x,z)-e_{1}\|_{L_x^{\infty}}<\frac{c}{1-c},
\end{equation}
where $0<c<1$. The estimate \eqref{gj} also holds for $\mu_{+}(x,z)$.

Let us first consider the case $0<c\leq\frac{1}{2}$, which implies that there is $\frac{c}{1-c}\leq1$.
The integral expression \eqref{a},\eqref{b} for the scattering data $a(k), b(k)$ and the relation \eqref{2.51} show that
\begin{equation*}
a(k)=1-\frac{1}{2i}\int_{\mathbb{R}}[u(x)v(x)\mu_{-}^{(1)}(x,z)-u(x)\mu_{-}^{(2)}(x,z)]dx,
\end{equation*}
\begin{equation*}
k^{-1}b(k)=\int_{\mathbb{R}}v(x)\mu_{-}^{(1)}(x,z)e^{-2ik^2x}dx.
\end{equation*}
Hence, for every $\text{Im}z\geq0$
\begin{equation}\label{a+}
\begin{aligned}
|a(k)|=&1-|\frac{1}{2i}\int_{\mathbb{R}}[u(x)v(x)\mu_{-}^{(1)}(x,z)-u(x)\mu_{-}^{(2)}(x,z)]dx|\\
\geq&1-\frac{1}{2}(\|uv\|_{L^{1}}\|\mu_{-}^{(1)}(x,z)-1\|_{L_{x}^{\infty}}+\|uv\|_{L^{1}}-\|u\|_{L^{1}}\|\mu_{-}^{(2)}(x,z)\|_{L_{x}^{\infty}})\\
\geq&1-\|Q_{1}(u)\|_{L^{1}}\\
>&1-c,
\end{aligned}
\end{equation}
and
\begin{align*}
|k^{-1}b(k)|\leq&\int_{\mathbb{R}}|v(x)\mu_{-}^{(1)}(x,z)|dx
\leq\|\mu_{-}^{(1)}(x,z)\|_{L_{x}^{\infty}}\|v\|_{L^{1}}\\
\leq&\frac{2\|Q_{1}(u)\|_{L^{1}}}{1-c}
<\frac{2c}{1-c}.
\end{align*}
From \eqref{2.70}, we have
\begin{align*}
|kb(k)|=&\frac{1}{2}|\mu_{+}^{(1)}(x,z)\mu_{-}^{(2)}(x,z)-\mu_{+}^{(2)}(x,z)\mu_{-}^{(1)}(x,z)|\\
\leq &\frac{1}{2}[\|\mu_{+}^{(1)}(x,z)-1\|_{L_{x}^{\infty}}\|\mu_{-}^{(2)}(x,z)\|_{L_{x}^{\infty}}+\|\mu_{-}^{(2)}(x,z)\|_{L_{x}^{\infty}}\\
&-\|\mu_{+}^{(2)}(x,z)\|_{L_{x}^{\infty}}\|\mu_{-}^{(1)}(x,z)-1\|_{L_{x}^{\infty}}+\|\mu_{+}^{(2)}(x,z)\|_{L_{x}^{\infty}}]\\
<&\frac{c}{(1-c)^{2}}.
\end{align*}
The above estimates give
\begin{equation}\label{b+}
\begin{aligned}
|b(k)|^2\leq|k^{-1}b(k)||kb(k)|<\frac{2c^{2}}{(1-c)^{3}}.
\end{aligned}
\end{equation}
In order to make $|r_{1}(k)|<1$, i.e. $|b(k)|<|a(k)|$, we just need to choose the appropriate constant $c$ such that $\frac{2c^{2}}{(1-c)^{3}}<(1-c)^2$ and $0<c\leq\frac{1}{2}$.
Conclusion $|r_{1}(k)|<1$ holds, particularly, if $0<c\leq0.295$, but is violated for $0.296<c\leq\frac{1}{2}$.

The second case $\frac{1}{2}<c<1$ can be analysed in the same way, but unfortunately there is no $c$ satisfying the condition such that $|r_{1}(k)|<1$.
Thus when $u\in H^{1,1}(\mathbb{R})$ and the constraint \eqref{small+} is satisfied, we have $\|Q_{1}(u)\|_{L^{1}}\leq 0.259$ and hence $|r_{1}(k)|<1,$ for every $k\in \mathbb{R}\cup i\mathbb{R}$. It is easy  to get $\|Q_{1}(u)\|_{L^{1}}=\|Q_{2}(u)\|_{L^{1}}$. Therefore, the same result can be obtained similarly for $r_{2}(k)$ by using the integral expression \eqref{d} for the scattering data $d(k)$ and the relationship (iv) between $b(k)$ and $c(k)$ in the Corollary \ref{c2+}.
\end{proof}

\begin{corollary}\label{csmall}
If $u\in H^{1,1}(\mathbb{R})$ satisfies the
small norm restriction \eqref{small+}, then
\begin{align}\label{smallab}
\frac{1}{a(k)},\ \frac{1}{d(k)}, \ b(k),\ c(k) \in L_{z}^{\infty}(\mathbb{R}).
\end{align}
Furthermore, there exists a constant $c_{0}$ such that \begin{align}\label{smallr}
|r_{j}(k)|\leq c_{0}<1,\ \  \forall \ k\in \mathbb{R}\cup i\mathbb{R},\ \ j=1,2.
\end{align}
\end{corollary}
\begin{proof}
According to the \eqref{a+} and \eqref{b+} in the proof of the Proposition \ref{psmall} show that when $c=0.295$, then
\begin{align*}
|a(k)|>0.705,\ \ |b(k)|<0.7048.
\end{align*}
The same is true for $d(k)$ and $c(k)$.
Therefore, there exists a constant $c_{0}$ such that \eqref{smallr} holds.
\end{proof}

Before we continue, let us give two Propositions to illustrate the nature of the reflection coefficients $r_{j}(k), j=1,2$ and $r_{\pm}(z)$ by using the Lemma \ref{l4} and Corollary \ref{csmall}.
\begin{proposition}\label{p2}
If $u\in H^{2}(\mathbb{R}) \cap H^{1,1}(\mathbb{R})$ satisfies the condition \eqref{small+}, then $r_{1}(k), r_{2}(k)\in L^{2,1}_{z}(\mathbb{R})$.
\end{proposition}
\begin{proof}
By the definition \eqref{3.1} of $r_{1}(k)$ and with the help of H\"{o}lder's inequality, the following estimate can be introduced
\begin{align*}
\|r_{1}(k)\|_{L_{z}^{2,1}}^{2}=\|\langle z\rangle\frac{b(k)}{a(k)}\|_{L_{z}^{2}}^{2}\leq &\|\frac{1}{a(k)}\|_{L_{z}^{\infty}}^{2}\|\langle z\rangle kb(k)\langle z\rangle k^{-1}b(k)\|_{L_{z}^{1}}\\
\leq&\|\frac{1}{a(k)}\|_{L_{z}^{\infty}}^{2}\|kb(k)\|_{L_{z}^{2,1}}\|k^{-1}b(k)\|_{L_{z}^{2,1}}.
\end{align*}
From \eqref{2.66} and \eqref{smallab}, we have $r_{1}(k)\in L_{z}^{2,1}(\mathbb{R})$.  The same is true for  $r_{2}(k)$.
\end{proof}
\begin{proposition}\label{l6}
If $u\in H^{2}(\mathbb{R})\cap H^{1,1}(\mathbb{R})$ satisfies the condition \eqref{small+}, then $r_{\pm}(z)\in H_{z}^{1}(\mathbb{R})\cap L_{z}^{2,1}(\mathbb{R})$ and $kr_{\pm}(z),\ 2ikzr_{+}(z)\in L_{z}^{\infty}(\mathbb{R})$. Moreover, the map
\begin{equation}\label{3.21}
u\mapsto[r_-(z),r_+(z)]
\end{equation}
is Lipschitz continuous from $H^{2}(\mathbb{R})\cap H^{1,1}(\mathbb{R})$ to $H_{z}^{1}(\mathbb{R})\cap L_{z}^{2,1}(\mathbb{R})$.
\end{proposition}
\begin{proof}
This proof is based on the results of the Lemma \ref{l4} and Corollary \ref{csmall}. By the definition \eqref{3.17} of $r_{-}(z)$ we have
\begin{align*}
\|r_{-}(z)\|_{L_{z}^{2}}=\|\frac{2ikb(k)}{a(k)}\|_{L_{z}^{2}}\leq2\|\frac{1}{a(k)}\|_{L_{z}^{\infty}}\|kb(k)\|_{L_{z}^{2}},
\end{align*}
thus, $r_{-}(z)\in L_{z}^{2}(\mathbb{R})$.
According to the Sobolev imbedding theorem, the Sobolev space $H^{1}(\mathbb{R})$ is imbedded in the space $C^{0,\frac{1}{2}}(\mathbb{\overline{R}})$.
Hence,
\begin{align*}
a(k)-a_{\infty},\ kb(k)\in C^{0,\frac{1}{2}}(\mathbb{\overline{R}}),
\end{align*}
is given by \eqref{2.65}. It can be further verified that when $\frac{1}{a(k)}\in L_{z}^{\infty}(\mathbb{R})$, $\frac{1}{a(k)}$ also belongs to $C^{0,\frac{1}{2}}(\mathbb{\overline{R}})$, which in turn gives
\begin{align*}
r_{-}(z)=\frac{2ikb(k)}{a(k)}\in C^{0,\frac{1}{2}}(\mathbb{\overline{R}}).
\end{align*}
Based on the above results we can differentiate  $r_{-}(z)$ in $z$, then
\begin{align*}
\|\partial_{z}r_{-}(z)\|_{L_{z}^{2}}=&2\|\frac{\partial_{z}[kb(k)]a(k)-kb(k)\partial_{z}a(k)}{a^{2}(k)}\|_{L_{z}^{2}}\\
\leq&2\|\frac{1}{a(k)}\|_{L_{z}^{\infty}}\|\partial_{z}[kb(k)]\|_{L_{z}^{2}}+2\|\frac{1}{a(k)}\|_{L_{z}^{\infty}}^{2}\|kb(k)\|_{L_{z}^{\infty}}\|\partial_{z}a(k)\|_{L_{z}^{2}}.
\end{align*}
From \eqref{2.65} and the property that $H^{1}(\mathbb{R})$  is imbedded in $L^{\infty}(\mathbb{R})$, it follows that $r_{-}(z)\in H_{z}^{1}(\mathbb{R})$. The result $r_{-}(z)\in L_{z}^{2,1}(\mathbb{R})$ is obtained immediately from \eqref{2.66}, that is
\begin{align*}
\|r_{-}(z)\|_{L_{z}^{2,1}}=2\|\frac{kb(k)}{a(k)}\|_{L_{z}^{2,1}}\leq\|\frac{1}{a(k)}\|_{L_{z}^{\infty}}\|kb(k)\|_{L_{z}^{2,1}}.
\end{align*}
The proof of $r_{-}(z)\in H_{z}^{1}(\mathbb{R})\cap L_{z}^{2,1}(\mathbb{R})$ has been completed. The same is true for $r_{+}(z)$.

Using the above results, it is possible to obtain
\begin{align*}
|kr_{\pm}(z)|^{2}=|zr_{\pm}^{2}(z)|=&|\int_{0}^{z}r_{\pm}^{2}(s)+2sr_{\pm}(s)r'_{\pm}(s)ds|\\
\leq&\|r_{\pm}\|_{L_{z}^{2}}^{2}+\|r'_{\pm}\|_{L_{z}^{2}}^{2}+\|r_{\pm}\|_{L_{z}^{2,1}}^{2}\\
\leq&\|r_{\pm}(z)\|_{H_{z}^1\cap L_{z}^{2,1}}^{2},
\end{align*}
thus $kr_{\pm}(z)\in L_{z}^{\infty}(\mathbb{R})$. We can define $\hat{r}_{-}(z):=\frac{2ikc(k)}{d(k)}$, and a similar proof as above gives $2ikzr_{+}(z)=-k\hat{r}_{-}(z)\in L_{z}^{\infty}(\mathbb{R})$.

We define the scattering date $\tilde{a}, \tilde{b}$ and reflection coefficient $\tilde{r}_{-}$ corresponding to potential $\tilde{u}$, then
\begin{align*}
r_{-}-\tilde{r}_{-}=\frac{2ik(b-\tilde{b})}{a}+\frac{2ik\tilde{b}}{a\tilde{a}}[(\tilde{a}-\tilde{a}_{\infty})-(a-a_{\infty})]+\frac{2ik\tilde{b}}{a\tilde{a}}(\tilde{a}_{\infty}-a_{\infty}).
\end{align*}
The Lipschitz continuity of the map \eqref{3.21} can be given by the above expression and Lemma \ref{l4}. The corresponding assertion for $r_{+}(z)$ can be proved similarly.
\end{proof}

\section{Inverse scattering transform}
\subsection{Cauchy operator and the solvability of the RH problem}

In fact, for any column vector $f\in\mathbb{C}^{2}$, we have
\begin{align*}
\text{Re}[f^{H}(I+S(x,k))f]=&\frac{1}{2}[f^{H}(I+S(x,k))f+f^{H}(I+S^{H}(x,k))f]\\
=&f^{H}(I+S_{H}(x,k))f.
\end{align*}
Therefore, the following very important conclusion is drawn in terms of investigating the solvability of the RH problem.
\begin{proposition}\label{p5}
If the reflection coefficients $r_{1}(k), r_{2}(k)$ satisfy the condition \eqref{smallr}, then there exists positive constants $\eta_{-}$ and $\eta_{+}$ for any $x\in\mathbb{R}$ and any column vector $f\in\mathbb{C}^2$ such that
\begin{align}\label{4.7}
\text{Re}[f^{H}(I+S(x,k))f]\geq \eta_{-}f^{H}f,\ \ k\in\mathbb{R}\cup i\mathbb{R},
\end{align}
\begin{align}\label{4.8}
\|(I+S(x,k))f\|\leq \eta_{+}\|f\|,\ \ k\in\mathbb{R}\cup i\mathbb{R}.
\end{align}
where $\|f\|$ denotes the Euclidean norm of vectors in $\mathbb{C}^2$.
\end{proposition}
\begin{proof}
When the reflection coefficient $r_{1}(k), r_{2}(k)$ satisfy \eqref{smallr}, it follows from \eqref{pd} that $I+S_{H}(x,k)$ is a positive definite matrix. Thus there exists a unitary
matrix $U$ such that
\begin{align*}
U^{H}(I+S_{H}(x,k))U=\text{diag}[\lambda_{-}(k),\lambda_{+}(k)],
\end{align*}
where $\lambda_{-}(k),\lambda_{+}(k)$ are the eigenvalues of the matrix $I+S_{H}(x,k)$ and have the following expressions
\begin{align*}
\lambda_{\pm}(k)=\frac{2-\text{Re}[r_{1}(k)r_{2}(k)]\pm\sqrt{\text{Re}^{2}[r_{1}(k)r_{2}(k)]+|r_{1}(k)-\overline{r}_{2}(k)|^{2}}}{2}.
\end{align*}
Since there exists $c_{0}$ that makes $|r_{j}(k)|\leq c_{0}<1$, there exists a constant $\eta_{-}>0$ such that
\begin{align*}
\lambda_{-}(k)=\frac{4-|r_{1}+\overline{r}_{2}|^{2}}{4-2\text{Re}[r_{1}r_{2}]+2\sqrt{\text{Re}^{2}[r_{1}r_{2}]+|r_{1}-\overline{r}_{2}|^{2}}}
\geq \eta_{-}.
\end{align*}
We set that $U^{H}f$ is equal to $g=[g_{1},g_{2}]^{T}$, then
\begin{align*}
\text{Re}[f^{H}(I+S(x,k))f]=&f^{H}(I+S_{H}(x,k))f =\lambda_{-}(k)|g_{1}|^{2}+\lambda_{+}(k)|g_{2}|^{2}\\
\geq&\eta_{-}g^{H}g
=\eta_{-}f^{H}f.
\end{align*}

Let $f=[f_{1},f_{2}]^{T}$, then
\begin{align*}
(I+S(x,k))f=\begin{bmatrix}(1-r_1(k)r_2(k))f_{1}-r_2(k)e^{-2ik^2x}f_{2}\\r_1(k)e^{2ik^2x}f_{1}+f_{2}\end{bmatrix}.
\end{align*}
It is clear from the calculation that when $r_{1}(k), r_{2}(k)$ satisfy \eqref{smallr}, there exists $\eta_{+}$ such that
\begin{equation*}
\|(I+S(x,k))f\|^{2}\leq\eta_{+}^{2}(|f_{1}|^{2}+|f_{2}|^{2})\leq \eta_{+}^{2}\|f\|^{2}.
\end{equation*}
\end{proof}

For a given function $f\in L^{p}(\mathbb{R})$ with $1\leq p<\infty$, the Cauchy operator $\mathcal{C}$ is defined by
\begin{equation}
\mathcal{C}(f)(z):=\frac{1}{2\pi i}\int_{\mathbb{R}}\frac{f(s)}{s-z}ds,\ \ z\in\mathbb{C}\setminus\mathbb{R}.
\end{equation}
When $z\pm\varepsilon i \ (z\in\mathbb{R}, \epsilon>0)$ approach to a point on the real axis  $z\in\mathbb{R}$ transversely from the upper
and the lower half planes respectively, the Cauchy operator $\mathcal{C}$ becomes the Cauchy projection operators $\mathcal{C}^{\pm}$ defined respectively by
\begin{equation}
\mathcal{C}^{\pm}(f)(z):=\lim_{\varepsilon\rightarrow 0}\frac{1}{2\pi i}\int_{\mathbb{R}}\frac{f(s)}{s-(z\pm\varepsilon i)}ds,\ \ z\in\mathbb{R}.
\end{equation}
If a function $f$ in $L^{p}(\mathbb{R}), \ 1\leq p<\infty$, the Hilbert transform $\mathcal{H}$ is defined by
\begin{align}
\mathcal{H}(f)(z):=\frac{1}{\pi }\lim_{\varepsilon\rightarrow 0}(\int_{-\infty}^{z-\varepsilon}+\int_{z+\varepsilon}^{\infty})\frac{f(s)}{s-z}ds,\ \ z\in\mathbb{R}.
\end{align}
For every $1<p<\infty$, $\mathcal{H}$ is a bounded operator form $L^{p}(\mathbb{R})$ to $L^{p}(\mathbb{R})$.

Some important properties of the Cauchy operator $\mathcal{C}$ and the Cauchy projection operators $\mathcal{C}^{\pm}$ are enumerated by references \cite{Dur, Duo}, see the following proposition.
\begin{proposition}\label{p4}
For every $f\in L^{p}(\mathbb{R}),\ 1\leq p<\infty$, the Cauchy operator $\mathcal{C}$ admits the following properties:
\begin{enumerate}[label=(\roman*)]
  \item $\mathcal{C}(f)(z)$ is analytic in $z\in\mathbb{C}\setminus\mathbb{R}$.
  \item $\mathcal{C}(f)(z)\rightarrow 0, \ \ |z|\rightarrow\infty$.
  \item In particular, if $f\in L^{1}(\mathbb{R})$,  then
      \begin{align}\label{4.5}
      \lim_{|z|\rightarrow\infty}z \mathcal{C}(f)(z)=-\frac{1}{2\pi i}\int_{\mathbb{R}}f(s)ds, \ \ z\in\mathbb{C}\setminus\mathbb{R}.
      \end{align}
\end{enumerate}
For every $f\in L^{p}(\mathbb{R}),\ 1< p<\infty$, the Cauchy projection operators $\mathcal{C}^{\pm}$ admit the following properties:
\begin{enumerate}[label=(\roman*)]
  \item There exists a positive constant $C_{p}$ (with $C_{2}=1$) such that
      \begin{align}\label{4.4}
      \|\mathcal{C}^{\pm}(f)\|_{L^p}\leq C_p\|f\|_{L^p}.
      \end{align}
  \item $\mathcal{C}^{\pm}(f)(z)=\pm\frac{1}{2}f(z)-\frac{i}{2}\mathcal{H}(f)(z),\
      \ z\in\mathbb{R}.$
  \item $\mathcal{C}^{+}-\mathcal{C}^{-}=I$,\ \ $\mathcal{C}^{+}+\mathcal{C}^{-}=-i\mathcal{H}$.
\end{enumerate}
\end{proposition}

In the following we consider the solvability of the RH Problem \ref{RH4} in the space of $L_{z}^{2}(\mathbb{R})$. It follows from Proposition \ref{p4} that if $N_{j,-}(x,k)\in L_{z}^{2}(\mathbb{R})$ is a solution of the Fredholm integral equation
\begin{align}\label{4.11}
N_{j,-}(x,k)=\mathcal{C}^{-}(N_{j,-}(x,k)S(x,k)+F_{j}(x,k))(z),\ \ j=1,2,\  z\in\mathbb{R},
\end{align}
then $N_{j,+}(x,k)\in L_{z}^{2}(\mathbb{R})$ can be obtained from the projection equation
\begin{align}\label{4.12}
N_{j,+}(x,k)=\mathcal{C}^{+}(N_{j,-}(x,k)S(x,k)+F_{j}(x,k))(z),\ \ j=1,2,\ z\in\mathbb{R}.
\end{align}
A further application of Proposition \ref{p4} yields that for any $x\in\mathbb{R}$, the RH Problem \ref{RH4} has a solution given by the Cauchy operator
\begin{align}\label{4.10}
N_{j}(x,k)=\mathcal{C}(N_{j,-}(x,k)S(x,k)+F_{j}(x,k))(z),\ \ j=1,2,\ z\in\mathbb{C}\setminus\mathbb{R}.
\end{align}

The integral equation \eqref{4.11} can be equivalently rewritten as
\begin{align}\label{4.13}
(I-\mathcal{C}^{-}_{S})N_{j,-}(x,k)=\mathcal{C}^{-}(F_{j})(x,k),\ \ j=1,2,\ k\in\mathbb{R}\cup i\mathbb{R},
\end{align}
where $\mathcal{C}^{-}_{S}N_{j,-}:=\mathcal{C}^-(N_{j,-}S)$.
We illustrate the solvability of the integral equation \eqref{4.11} in Proposition \ref{p78} below by considering its equivalent form \eqref{4.13}. It should be noted that the subscript $j$ takes either 1 or 2, and we will not emphasize this point in the following analysis.
\begin{proposition}\label{p78}
If the reflection coefficients $r_{1}(k), r_{2}(k)\in L_{z}^{2}(\mathbb{R})$ satisfy the condition \eqref{smallr} and $r_{\pm}(z)\in L_{z}^{2}(\mathbb{R})$, then $\forall x\in\mathbb{R}$, there exists a unique solution $N_{j,-}(x,k)\in L_{z}^{2}(\mathbb{R})$ of the linear inhomogeneous equation \eqref{4.13}. Moreover, the inverse operator $(I-\mathcal{C}^{-}_{S})^{-1}$ exists and is a bounded operator from $L_{z}^{2}(\mathbb{R})$ to $L_{z}^{2}(\mathbb{R})$.
\end{proposition}
\begin{proof}
By the definition of $S(x,k)$ \eqref{3.7} and the definition of $F_{j}(x,k)$ \eqref{3.28}, it is immediate that $\forall x\in\mathbb{R}$, $S(x,k)\in L^{2}_{z}(\mathbb{R})\cap L^{\infty}_{z}(\mathbb{R})$ and $F_{j}(x,k)\in L^{2}_{z}(\mathbb{R})$. According to references \cite{Bea,Bea1,Zho}, it is known that the operator $I-\mathcal{C}^{-}_{S}$ is a Fredholm operator of the index zero. So we only need to prove that the zero solution to the equation
\begin{equation}\label{4.13+}
(I-\mathcal{C}^{-}_{S})f=0
\end{equation}
is unique in $L_{z}^{2}(\mathbb{R})$.

Suppose that there exists a nonzero row vector solution $f\in L_{z}^{2}(\mathbb{R})$ for \eqref{4.13+}, then we define
\begin{align*}
f_{1}(z):=\mathcal{C}(fS)(z),\ \ f_{2}(z):=[\mathcal{C}(fS)(\overline{z})]^{H}.
\end{align*}
By the property (i) of Cauchy operator in Proposition \ref{p4}, $f_{1}(z), f_{2}(z)$  are analytic functions in  $\mathbb{C}\setminus\mathbb{R}$. Making a semicircle $C_{R}$ in the upper half plane $\mathbb{C}^{+}$ with the zero as the centre and $R$ as the radius, we have
\begin{align*}
\int_{-R}^{R} f_{1}(s)f_{2}(s)ds=\int_{C_{R}}f_{1}(s)f_{2}(s)ds
\end{align*}
by the Cauchy-Goursat theorem. Let $R\rightarrow \infty$ and using the limit \eqref{4.5} and the property (iii) of Cauchy projection operators in Proposition \ref{p4} we get
\begin{align*}
0=&\int_{\mathbb{R}}f_{1}(s)f_{2}(s)ds\\
=&\int_{\mathbb{R}}\mathcal{C}^{+}(fS)(s)[\mathcal{C}^{-}(fS)(\overline{s})]^{H}ds\\
=&\int_{\mathbb{R}}[\mathcal{C}^{-}(fS)(s)+fS(s)][\mathcal{C}^{-}(fS)(s)]^{H}ds\\
=&\int_{\mathbb{R}}f(s)(I+S)f(s)^{H}ds.
\end{align*}
The above equality contradicts the result \eqref{4.7}, so the linear equation \eqref{4.13} has a unique solution $N_{j,-}(k)$ in $L_{z}^{2}(\mathbb{R})$.

In the following we prove that the operator $I-\mathcal{C}^{-}_{S}$ is invertible in the space $L_{z}^{2}(\mathbb{R})$. For every row vector $g\in L_{z}^{2}(\mathbb{R})$, we write
\begin{equation}\label{4.15}
(I-\mathcal{C}^{-}_{S})g=G.
\end{equation}
From \eqref{4.4} it is easy to see that $G\in L_{z}^{2}(\mathbb{R})$.
Suppose $g$ can be decomposed as $g=g_{+}-g_{-}$, from the property (iii) of Cauchy projection operators in Proposition \ref{p4} we know that $g_{+}, g_{-}$ satisfy the following two equations respectively
\begin{equation}\label{4.16}
g_{-}-\mathcal{C}^{-}(g_{-}S)=\mathcal{C}^{-}(G),\ \ g_{+}-\mathcal{C}^{-}(g_{+}S)=\mathcal{C}^{+}(G).
\end{equation}
From the above proof and the fact that $\mathcal{C}^{-}(G),\mathcal{C}^{+}(G) \in L_{z}^{2}(\mathbb{R})$, it follows that the decomposition $g=g_{+}-g_{-}$ exists and is unique.

We first consider $g_{-}$. Two analytic functions in $\mathbb{C}\setminus\mathbb{R}$ are defined by
\begin{align*}
h_{1}(z):=\mathcal{C}(g_{-}S)(z),\ h_{2}(z):=[\mathcal{C}(g_{-}S+G)(\overline{z})]^{H}.
\end{align*}
Similarly, we make a semicircle in $\mathbb{C}^{+}$ with zero as its centre and $R$ as its radius, then letting $R\rightarrow\infty$ gives
\begin{align*}
0=&\int_{\mathbb{R}} h_{1}(z)h_{2}(z)dz\\
=&\int_{\mathbb{R}}\mathcal{C}^{+}(g_{-}S)[\mathcal{C}^{-}(g_{-}S+G)]^{H}dz\\
=&\int_{\mathbb{R}}[\mathcal{C}^{-}(g_{-}S)+g_{-}S][\mathcal{C}^{-}(g_{-}S+G)]^{H}dz\\
=&\int_{\mathbb{R}}[g_{-}-\mathcal{C}^{-}(G)+g_{-}S]g_{-}^{H}dz,
\end{align*}
which implies that
\begin{equation*}
\int_{\mathbb{R}}g_{-}(I+S)g_{-}^{H}dz=\int_{\mathbb{R}}\mathcal{C}^{-}(G)g_{-}^{H}dz.
\end{equation*}
According to the estimated equations \eqref{4.7} and \eqref{4.4}, there is a positive constant $\eta_{-}$ such that
\begin{equation*}
\eta_{-}\|g_{-}\|_{L_{z}^{2}}^{2}\leq\text{Re}\int_{\mathbb{R}}g_{-}(I+S)g_{-}^{H}dz=\text{Re}\int_{\mathbb{R}}\mathcal{C}^{-}(G)g_{-}^{H}dz\leq\|G\|_{L_{z}^{2}}\|g_{-}\|_{L_{z}^{2}}.
\end{equation*}
Therefore we have
\begin{equation}\label{4.17}
\|(I-\mathcal{C}^{-}_{S})^{-1}\mathcal{C}^{-}(G)\|_{L^2_z}=\|g_{-}\|_{L^2_z}\leq \eta_{-}^{-1}\|G\|_{L^2_z}.
\end{equation}

Next we consider $g_{+}$ and use the property $\mathcal{C}^{+}-\mathcal{C}^{-}=I$ of the Cauchy projection operators to rewrite the second equation of \eqref{4.16} in the following form
\begin{equation*}
g_{+}(I+S)-\mathcal{C}^{+}(g_{+}S)=\mathcal{C}^{+}(G).
\end{equation*}
This time, we make a semicircle in the lower half plane $\mathbb{C}^{-}$, and using similar means as above we can obtain
\begin{equation*}
\int_{\mathbb{R}}g_{+}(I+S)^{H}g_{+}^{H}dz=\int_{\mathbb{R}}\mathcal{C}^{+}(G)(I+S)^{H}g_{+}^{H}dz.
\end{equation*}
According to the estimated equations \eqref{4.7} and \eqref{4.8}, there are positive constants $\eta_{-}, \eta_{+}$ such that
\begin{align*}
\eta_{-}\|g_{+}\|_{L_{z}^{2}}^{2}\leq&\text{Re}\int_{\mathbb{R}}g_{+}(I+S)^{H}g_{+}^{H}dz=\text{Re}\int_{\mathbb{R}}\mathcal{C}^{+}(G)(I+S)^{H}g_{+}^{H}dz\\
\leq& \eta_{+}\|G\|_{L_{z}^{2}}\|g_{+}\|_{L_{z}^{2}}.
\end{align*}
So we have
\begin{equation}\label{4.19}
\|(I-\mathcal{C}^{-}_{S})^{-1}\mathcal{C}^{+}(G)\|_{L^2_z}=\|g_{+}\|_{L^2_z}\leq \eta_{-}^{-1}\eta_{+}\|G\|_{L^2_z}.
\end{equation}
Combining \eqref{4.17} with \eqref{4.19} it follows that there exists a positive constant $\eta$ such that
\begin{equation}\label{4.15}
\|(I-\mathcal{C}^{-}_{S})^{-1}G\|_{L^2_z}\leq \eta \|G\|_{L^2_z}.
\end{equation}
\end{proof}

From the above Proposition \ref{p78} it can be seen that \eqref{4.10} is a solution to the RH problem \ref{RH4}.
Therefore RH problem \ref{RH3} is also solvable according to the relationship \eqref{3.29} between the RH problem \ref{RH3} and \ref{RH4}.
By the Beals-Coifman theorem in \cite{Bea} it follows that the solution to the RH problem \ref{RH3} is unique, and hence the solution to the RH problem \ref{RH4} is also unique.

We consider mainly the RH problem \ref{RH3} in the $z$-plane. We note that
\begin{equation}\label{4.1} M_{\pm}(x,z):=[m_{\pm}(x,z),n_{\pm}(x,z)],
\end{equation}
and the superscripts $(1)$, $(2)$ used later represent the first and second columns of the square matrix, respectively. However, for column vectors the superscripts $(1)$ and $(2)$ denote the first and second rows of the vector, respectively.  Recalling the relation \eqref{3.29} as well as \eqref{3.27}, then we have
\begin{align*}
N_{1,\pm}(x,k)=&M_{\pm}(x,z)\rho_{1}(k)-\rho_{1}(k)\\
=&[m_{\pm}(x,z)-e_{1},2ik(n_{\pm}(x,z)-e_{2})].
\end{align*}
Moreover, equations \eqref{4.11}, \eqref{3.27-} and $F_{1}(x,k)=\rho_{1}(k)S(x,k)$ imply that
\begin{align*}
N_{1,\pm}(x,k)=&\mathcal{C}^{\pm}(N_{1,-}(x,k)S(x,k)+F_{1}(x,k))(z)\\
=&\mathcal{C}^{\pm}((M_{-}(x,z)-I)\rho_{1}(k)S(x,k)+\rho_{1}(k)S(x,k))(z)\\
=&\mathcal{C}^{\pm}(M_{-}(x,z)R(x,z)\rho_{1}(k))(z)\\
=&\mathcal{C}^{\pm}[(M_{-}(x,z)R(x,z))^{(1)},2ik(M_{-}(x,z)R(x,z))^{(2)}](z).
\end{align*}
Combining the above two expressions for $N_{1,\pm}(x,k)$, we can obtain
\begin{align}
m_{\pm}(x,z)-e_{1}&=\mathcal{C}^{\pm}((M_{-}(x,z)R(x,z))^{(1)})(z), \label{4.23} \\
2ik(n_{\pm}(x,z)-e_{2})&=\mathcal{C}^{\pm}(2ik(M_{-}(x,z)R(x,z))^{(2)})(z).\label{4.31}
\end{align}

Using the same means for $N_{2,\pm}(x,k)$, we have
\begin{align*}
N_{2,\pm}(x,k)=&M_{\pm}(x,z)\rho_{2}(k)-\rho_{2}(k)\\
=&[\frac{1}{2ik}(m_{\pm}(x,z)-e_{1}),n_{\pm}(x,z)-e_{2}],
\end{align*}
and
\begin{align*}
N_{2,\pm}(x,k)=&\mathcal{C}^{\pm}(N_{2,-}(x,k)S(x,k)+F_{2}(x,k))(z)\\
=&\mathcal{C}^{\pm}(M_{-}(x,z)R(x,z)\rho_{2}(k))(z)\\
=&\mathcal{C}^{\pm}[\frac{1}{2ik}(M_{-}(x,z)R(x,z))^{(1)},(M_{-}(x,z)R(x,z))^{(2)}](z).
\end{align*}
These results indicate that
\begin{align}
\frac{1}{2ik}(m_{\pm}(x,z)-e_{1})&=\mathcal{C}^{\pm}(\frac{1}{2ik}(M_{-}(x,z)R(x,z))^{(1)})(z), \label{4.32} \\
n_{\pm}(x,z)-e_{2}&=\mathcal{C}^{\pm}((M_{-}(x,z)R(x,z))^{(2)})(z).\label{4.26}
\end{align}

Notice that we can combine equations \eqref{4.23}, \eqref{4.26} and write them together in the form
\begin{equation}\label{4.27}
M_{\pm}(x,z)=I+\mathcal{C}^{\pm}(M_-(x,z)R(x,z))(z),\ \ z\in\mathbb{R}.
\end{equation}\label{4.28}
Thus the unique solution to RH problem \ref{RH3} can be expressed as
\begin{equation}\label{4.28}
M(x,z)=I+\mathcal{C}(M_-(x,z)R(x,z))(z),\ \ z\in\mathbb{C}\setminus\mathbb{R},
\end{equation}
where $M_{-}(x,z)R(x,z)$ can be rewritten as
\begin{equation}\label{MR}
M_{-}(x,z)R(x,z)=[n_{+}(x,z)r_{-}(z)e^{2izx},m_{-}(x,z)r_{+}(z)e^{-2izx}]
\end{equation}
by utilizing the equations \eqref{3.23}, \eqref{3.26} and \eqref{4.1}. Further, \eqref{4.31} and \eqref{4.32} can also be rewritten as
\begin{align}
2ik\mathcal{C}^{\pm}(m_{-}(x,z)r_{+}(z)e^{-2izx})(z)&=\mathcal{C}^{\pm}(2ikm_{-}(x,z)r_{+}(z)e^{-2izx})(z),\label{tk1}\\
\frac{1}{2ik}\mathcal{C}^{\pm}(n_{+}(x,z)r_{-}(z)e^{2izx})(z)&=\mathcal{C}^{\pm}(\frac{1}{2ik}n_{+}(x,z)r_{-}(z)e^{2izx})(z).\label{tk2}
\end{align}

By the relation \eqref{3.29} and Proposition \ref{p78}, the following corollary can easily be obtained.
\begin{corollary}\label{l9}
If the reflection coefficients $r_{1}(k), r_{2}(k)\in L_{z}^{2}(\mathbb{R})$ satisfy the condition \eqref{smallr} and $r_{\pm}(z)\in L_{z}^{2}(\mathbb{R})$, then the unique solution $M(x,z)$ to RH problem \ref{RH3} satisfies $M_{\pm}(x,z)-I\in L_{z}^{2}(\mathbb{R})$.
\end{corollary}

\subsection{Reconstruction formulas for the potential}
According to the solution \eqref{3.22} for the RH problem \ref{RH2} and the relation \eqref{3.25} it follows that
\begin{align*}
\Gamma_{+}(x,z)=[\frac{\mu_-(x,z)}{a(z)},\gamma_+(x,z)]=\Psi_{\infty}(x)M_{+}(x,z),\\
\Gamma_{-}(x,z)=[\mu_+(x,z),\frac{\gamma_{-}(x,z)}{d(z)}]=\Psi_{\infty}(x)M_{-}(x,z).
\end{align*}
By using the limits
\begin{equation*}
\lim_{|z|\rightarrow\infty}4z\gamma_{+}^{(1)}(x,z)=-u(x)\nu_{+}^{\infty}(x),\ \ \text{and} \ \ \lim_{|z|\rightarrow\infty}z\mu_{+}^{(2)}(x,z)=\alpha_{+}^{(2)}(x)
\end{equation*}
in Lemma \ref{l2}, we have
\begin{align}
\lim_{|z|\rightarrow\infty}zn_{+}^{(1)}(x,z)=&-\frac{1}{4}u(x)e^{i\int^{+\infty}_{x}u(y)v(y)dy},\label{4.70}\\
\lim_{|z|\rightarrow\infty}zm_{-}^{(2)}(x,z)=&-\frac{1}{2i}e^{\frac{1}{2i}\int^{+\infty}_{x}u(y)v(y)dy}\partial_{x}[v(x)e^{\frac{1}{2i}\int^{+\infty}_{x}u(y)v(y)dy}].\label{4.69}
\end{align}

Because $r_{\pm}(z)\in H_{z}^{1}(\mathbb{R})\cap L_{z}^{2,1}(\mathbb{R})$, it is easy to know that $\forall x\in \mathbb{R}, R(x,z)\in L_{z}^{1}(\mathbb{R})\cap L_{z}^{2}(\mathbb{R})$ and furthermore from $M_{\pm}(x,z)-I\in L_{z}^{2}(\mathbb{R})$ we know that $M_{\pm}(x,z)R(x,z)\in L_{z}^{1}(\mathbb{R})$. Thus for any $x\in \mathbb{R}$, it follows that
\begin{equation*}
\lim_{|z|\rightarrow\infty}z\mathcal{C}(M_{-}R)(z)=-\frac{1}{2\pi i}\int_{\mathbb{R}}M_{-}(x,z)R(x,z)dz
\end{equation*}
by \eqref{4.5}.
With the help of \eqref{MR}, we can rewrite \eqref{4.28} as
\begin{align}
\begin{bmatrix}
m_{\pm}^{(1)}-1&n_{\pm}^{(1)}\\
m_{\pm}^{(2)}&n_{\pm}^{(2)}-1
\end{bmatrix}
=\mathcal{C}
\begin{bmatrix}
n_{+}^{(1)}r_{-}e^{2izx}&m_{-}^{(1)}r_{+}e^{-2izx}\\
n_{+}^{(2)}r_{-}e^{2izx}&m_{-}^{(2)}r_{+}e^{-2izx}
\end{bmatrix},\ \ z\in\mathbb{C}\setminus\mathbb{R},
\end{align}
so that the reconstruction
formulas \eqref{4.70} and \eqref{4.69} can be rewritten as
\begin{align}
u(x)e^{i\int^{+\infty}_{x}u(y)v(y)dy}=\frac{2}{\pi i}\int_{\mathbb{R}}m_{-}^{(1)}(x,z)r_{+}(z)e^{-2izx}dz,\label{4.72}\\
e^{\frac{1}{2i}\int^{+\infty}_{x}u(y)v(y)dy}\partial_{x}[v(x)e^{\frac{1}{2i}\int^{+\infty}_{x}u(y)v(y)dy}]=\frac{1}{\pi }\int_{\mathbb{R}}n_{+}^{(2)}(x,z)r_{-}(z)e^{2izx}dz.\label{4.71}
\end{align}

In fact, we cannot get any more information from the RH problem \ref{RH2}, \ref{RH3} in the $z$-plane, and the reconstruction formulas \eqref{4.72} and \eqref{4.71} are not sufficient to recover the potential $u(x)$ on the whole line $\mathbb{R}$. So next we introduce the scalar RH problem $\delta(z)$ to transform the jump matrix $R(x,z)$ of the RH problem in the $z$-plane to obtain more information.

\begin{rhp}\label{RH5}
Find a scalar function $\delta(z)$ that satisfies the following conditions
\begin{itemize}
  \item Analyticity: $\delta(z)$  is analytical in $\mathbb{C}\setminus \mathbb{R}$.
  \item Jump condition:  $\delta(z)$ has continuous boundary values $\delta_{\pm}(z)$ on $\mathbb{R}$ and
  \begin{equation}\label{4.84}
  \delta_{+}(z)-\delta_{-}(z)=\delta_{-}(z)r_{+}(z)r_{-}(z),\ \ z\in\mathbb{R}.
  \end{equation}
  \item Asymptotic condition:
  \begin{equation}\label{4.84+}
  \delta(z)\rightarrow 1,\ \ |z|\rightarrow\infty.
  \end{equation}
\end{itemize}
\end{rhp}

It follows from the properties of the Cauchy projection operator that
\begin{equation}
\delta_{\pm}(z)=e^{\mathcal{C}^{\pm}(\log(1+r_{+}r_{-}))(z)},\ \ z\in\mathbb{R},
\end{equation}
satisfy the jump condition \eqref{4.84}, where $\log(1+r_{+}(z)r_{-}(z))\in\L_{z}^{2}(\mathbb{R})$, which is proved in the Proposition \ref{p9}. The analytic continuation of functions $\delta_{\pm}(z)$ in $\mathbb{C}\setminus\mathbb{R}$ can be expressed by the Cauchy operator
\begin{align}
\delta(z)=e^{\mathcal{C}(\log(1+r_{+}r_{-}))(z)},\ \ \mathbb{C}\setminus\mathbb{R},
\end{align}
which is the unique solution to RH problem \ref{RH5}.

We set
\begin{equation}
M_{\delta}(x,z):=M(x,z)\delta^{-\sigma_{3}}(z)=M(x,z)\begin{bmatrix}\delta^{-1}(z)&0\\0&\delta(z)\end{bmatrix},
\end{equation}
and
\begin{equation}\label{rdelta}
r_{\delta,+}(z):=\delta_{+}(z)\delta_{-}(z)r_{+}(z), \ \ r_{\delta,-}(z):=\delta_{+}^{-1}(z)\delta_{-}^{-1}(z)r_{-}(z), \ \ z\in\mathbb{R}.
\end{equation}
Then, $M_{\delta}(x,z)$ satisfies the new RH problem:

\begin{rhp}\label{RH6}
Find a matrix-valued function $M_{\delta}(x,z)$ that satisfies the following conditions
\begin{itemize}
  \item Analyticity: $M_{\delta}(x,z)$  is analytical in $\mathbb{C}\setminus \mathbb{R}$.
  \item Jump condition:  $M_{\delta}(x,z)$ has continuous boundary values $M_{\delta,\pm}(x,z)$ on $\mathbb{R}$ and
     \begin{equation}\label{4.88}
      M_{\delta,+}(x,z)-M_{\delta,-}(x,z)=M_{\delta,-}(x,z)R_{\delta}(x,z),\ \ z\in\mathbb{R},
      \end{equation}
      where
      \begin{align}
      R_{\delta}(x,z)=\begin{bmatrix}0&r_{\delta,+}(z)e^{-2izx}\\r_{\delta,-}(z)e^{2izx}&r_{\delta,+}(z)r_{\delta,-}(z)
      \end{bmatrix}.
       \end{align}
  \item Asymptotic condition:
       \begin{equation}
       M_{\delta}(x,z)\rightarrow I,\ \ |z|\rightarrow\infty.
       \end{equation}
\end{itemize}
\end{rhp}

In the next analysis, we define
\begin{equation*}
\widehat{f}(\lambda)=\frac{1}{2\pi}\int_{\mathbb{R}}f(\xi)e^{-i\lambda\xi}d\xi
\end{equation*}
as the Fourier transform of the function $f$, then
\begin{equation*}
\|f\|_{L^{2}}^{2}=2\pi\|\widehat{f}\|_{L^{2}}^{2},\ \ \|f\|_{H^{1}}=\sqrt{2\pi}\|\widehat{f}\|_{L^{2,1}}.
\end{equation*}
\begin{proposition}\label{p9}
If $r_{\pm}(z)\in H_{z}^{1}(\mathbb{R})\cap L_{z}^{2,1}(\mathbb{R})$ and $r_{j}(k),\  j=1,2$ satisfy the condition \eqref{smallr}, then $r_{\delta,\pm}(z)\in H_{z}^{1}(\mathbb{R})\cap L_{z}^{2,1}(\mathbb{R})$.
\end{proposition}
\begin{proof}
We first prove that $\log(1+r_{+}(z)r_{-}(z))\in H_{z}^{1}(\mathbb{R})$. Let $\omega=r_{+}(z)r_{-}(z)$, then $\omega\in\mathbb{C}$ and $|\omega|=|r_{1}r_{2}|\leq c_{0}^{2}<1$. Since $\text{Log} \omega$ can define a single-valued branch satisfying $\log 1=0$ in $\omega\in\mathbb{C}\setminus(-\infty,0]$, so $f(\omega)=\log(1+\omega)$ is analytic in $\omega\in\mathbb{C}\setminus(-\infty,-1]$ and $f'(\omega)$ is bounded on the disc $B(0,c_{0}^{2})$. It follows from the complex mean value theorem
that $\forall z_{1}, z_{2}\in\mathbb{R}$, $\omega_{1}=r_{+}(z_{1})r_{-}(z_{1})$, $\omega_{2}=r_{+}(z_{2})r_{-}(z_{2})$, there exists two points $\xi_{1},\xi_{2}$ on the segment of the line joining $\omega_{1}, \omega_{2}$ such that
\begin{equation*}
|f(\omega_{1})-f(\omega_{2}))|\leq 2\sup_{\omega\in B(0,c_{0}^{2})}|f'(\omega)||\omega_{1}-\omega_{2}|.
\end{equation*}
Because $r_{+}(0)r_{-}(0)=-r_{1}(0)r_{2}(0)=0$, it follows that
\begin{equation*}
|\log(1+r_{+}(z)r_{-}(z))|\leq C|r_{+}(z)r_{-}(z)|,
\end{equation*}
where $C$ is a constant associated with $c_{0}^{2}$.
Therefore, we have $\log(1+r_{+}r_{-})\in L_{z}^{2}(\mathbb{R})$. It is easy to prove $\partial_{z}\log(1+r_{+}r_{-})\in L_{z}^{2}(\mathbb{R})$, so $\log(1+r_{+}r_{-})\in H_{z}^{1}(\mathbb{R})$.

Next we prove that $\delta_{\pm}(z)\in L_{z}^{\infty}(\mathbb{R})$. Based on Fourier theory, we rewrite the following integral equation
\begin{align*}
\mathcal{C}^{+}(\log(1+r_{+}r_{-}))(z)=&\lim_{\varepsilon\rightarrow 0}\frac{1}{2\pi i}\int_{\mathbb{R}}\frac{\log(1+r_{+}(s)r_{-}(s))}{s-(z+\varepsilon i)}ds\\
=&\lim_{\varepsilon\rightarrow 0}\frac{1}{2\pi i}\int_{\mathbb{R}}\int_{\mathbb{R}}\widehat{\log(1+r_{+}r_{-})}(\lambda)e^{i\lambda s}d\lambda\frac{1}{s-(z+\varepsilon i)}ds\\
=&\lim_{\varepsilon\rightarrow  0}\int_{\mathbb{R}}\widehat{\log(1+r_{+}r_{-})}(\lambda)\frac{1}{2\pi i}\int_{\mathbb{R}}\frac{e^{i\lambda s}}{s-(z+\varepsilon i)}dsd\lambda.
\end{align*}
Notice that the integrand function of the integral
\begin{align*}
\frac{1}{2\pi i}\int_{\mathbb{R}}\frac{e^{i\lambda s}}{s-(z+\varepsilon i)}ds
\end{align*}
has first order pole $z+i\varepsilon\in\mathbb{C}^{+}$. When $\lambda>0$, we make a sufficiently large semicircle $C_{R}$ of radius $R$ centered at zero in $\mathbb{C}^{+}$
such that $z+i\varepsilon$ is included in the semicircular disk $D_{R}$. By using the residue theorem, we have
\begin{align*}
\frac{1}{2\pi i}\int_{-R}^{R}\frac{e^{i\lambda s}}{s-(z+\varepsilon i)}ds+\frac{1}{2\pi i}\int_{C_{R}}\frac{e^{i\lambda s}}{s-(z+\varepsilon i)}ds=e^{i\lambda(z+i\varepsilon)},\ \ \lambda>0.
\end{align*}
Let $R\rightarrow\infty$, we have
\begin{align*}
\frac{1}{2\pi i}\int_{\mathbb{R}}\frac{e^{i\lambda s}}{s-(z+\varepsilon i)}ds=e^{i\lambda(z+i\varepsilon)},\ \ \lambda>0
\end{align*}
by Jordan theorem. Similarly, when $\lambda<0$, we have
\begin{align*}
\frac{1}{2\pi i}\int_{\mathbb{R}}\frac{e^{i\lambda s}}{s-(z+\varepsilon i)}ds=0,\ \ \lambda<0.
\end{align*}
Therefore, we have
\begin{align*}
&\|\mathcal{C}^{+}(\log(1+r_{+}r_{-}))(z)\|_{L_{z}^{\infty}}=\|\int_{0}^{+\infty}\widehat{\log(1+r_{+}r_{-})}(\lambda)e^{i\lambda z}d\lambda\|_{L_{z}^{\infty}}\\
\leq &\|\widehat{\log(1+r_{+}r_{-})}\|_{L^{1}}
\leq\sqrt{\pi}\|\widehat{\log(1+r_{+}r_{-})}\|_{L^{2,1}}
=\frac{1}{\sqrt{2}}\|\log(1+r_{+}r_{-})\|_{H_{z}^{1}},
\end{align*}
which means that
\begin{equation*}
\delta_{+}(z)=e^{\mathcal{C}^{+}(\log(1+r_{+}r_{-}))(z)}\in L_{z}^{\infty}(\mathbb{R}).
\end{equation*}
A similar proof can be given for $\delta_{-}(z)\in L_{z}^{\infty}(\mathbb{R})$.

Finally, we prove that the conclusion $r_{\delta,\pm}(z)\in H_{z}^{1}(\mathbb{R})\cap L_{z}^{2,1}(\mathbb{R})$.
We prove the statement for $r_{\delta,+}(z)$ only. The proof for $r_{\delta,-}(z)$ is analogous. From the definition \eqref{rdelta} of $r_{\delta,+}(z)$ and the conclusions $\delta_{\pm}(z)\in L_{z}^{\infty}(\mathbb{R})$ and $r_{+}(z)\in L_{z}^{2,1}(\mathbb{R})$, it immediately follows that $r_{\delta,+}(z)\in L_{z}^{2,1}(\mathbb{R})$.
By the property (iii) of the Cauchy projection operator in the Proposition \ref{p4} we get
\begin{equation*}
\delta_{+}(z)\delta_{-}(z)=e^{-i\mathcal{H}(\log(1+r_{+}r_{-}))(z)}.
\end{equation*}
By the properties of the Hilbert transform $\mathcal{H}$, we have
\begin{align*}
\|\partial_{z}\mathcal{H}(\log(1+r_{+}r_{-}))\|_{L_{z}^{2}}=\|\mathcal{H}(\partial_{z}\log(1+r_{+}r_{-}))\|_{L_{z}^{2}}\leq\|\partial_{z}\log(1+r_{+}r_{-})\|_{L_{z}^{2}},
\end{align*}
and therefore $r_{\delta,+}(z)\in H_{z}^{1}(\mathbb{R})$.

\end{proof}
From the Proposition \ref{p9} it follows that the unique solution to RH problem \ref{RH6} can be represented by the Cauchy operator
\begin{equation}
M_{\delta}(x,z)=I+\mathcal{C}(M_{\delta,-}(x,\cdot)R_{\delta}(x,\cdot))(z),\ \ z\in\mathbb{C}\setminus \mathbb{R},
\end{equation}
which projections on the $\mathbb{R}$ are
\begin{equation}\label{4.27+}
M_{\delta,\pm}(x,z)=I+\mathcal{C}^{\pm}(M_{\delta,-}(x,\cdot)R_{\delta}(x,\cdot))(z),\ \ z\in\mathbb{R}.
\end{equation}
Similarly, we denote
\begin{equation*}
M_{\delta,\pm}(x,z):=[m_{\delta,\pm}(x,z),n_{\delta,\pm}(x,z)],
\end{equation*}
then $M_{\delta,-}(x,z)R_{\delta}(x,z)$ can be rewritten as
\begin{equation}\label{MR+}
M_{\delta,-}(x,z)R_{\delta}(x,z)=[n_{\delta,-}(x,z)r_{\delta,-}(z)e^{2izx},m_{\delta,+}(x,z)r_{\delta,+}(z)e^{-2izx}].
\end{equation}
We have
\begin{align}
\lim_{|z|\rightarrow\infty}zn_{\delta,+}^{(1)}(x,z)=&-\frac{1}{4}u(x)e^{i\int^{+\infty}_{x}u(y)v(y)dy},\label{4.70+}\\
\lim_{|z|\rightarrow\infty}zm_{\delta,-}^{(2)}(x,z)=&-\frac{1}{2i}e^{\frac{1}{2i}\int^{+\infty}_{x}u(y)v(y)dy}\partial_{x}[v(x)e^{\frac{1}{2i}\int^{+\infty}_{x}u(y)v(y)dy}],\label{4.69+}
\end{align}
from \eqref{4.70} and \eqref{4.69}, where \eqref{4.84+} is used. Since $\delta_{\pm}(z)\in L_{z}^{\infty}(\mathbb{R})$ and $r_{\delta,\pm}(z)\in H_{z}^{1}(\mathbb{R})\cap L_{z}^{2,1}(\mathbb{R})$, we know that $\forall x\in \mathbb{R}, R_{\delta}(x,z)\in L_{z}^{1}(\mathbb{R})\cap L_{z}^{2}(\mathbb{R})$ and $M_{\delta, -}(x,z)R_{\delta}(x,z)\in L_{z}^{1}(\mathbb{R})$. Then the formulas \eqref{4.70+} and \eqref{4.69+} can be rewritten as
\begin{align}
u(x)e^{i\int^{+\infty}_{x}u(y)v(y)dy}&=\frac{2}{\pi i}\int_{\mathbb{R}}m_{\delta,+}^{(1)}(x,z)r_{\delta,+}(z)e^{-2izx}dz,\label{4.91}\\
e^{\frac{1}{2i}\int^{+\infty}_{x}u(y)v(y)dy}\partial_{x}[v(x)e^{\frac{1}{2i}\int^{+\infty}_{x}u(y)v(y)dy}]&=\frac{1}{\pi}\int_{\mathbb{R}}n_{\delta,-}^{(2)}(x,z)r_{\delta,-}(z)e^{2izx}dz.\label{4.90}
\end{align}
according to the character \eqref{4.5} for Cauchy operator $\mathcal{C}$ in Proposition \ref{p4}.

\subsection{Estimates of the potential}
We shall estimate the potential $u(x)$ with the help of two sets of reformulation formulas \eqref{4.72}, \eqref{4.71} and \eqref{4.91}, \eqref{4.90}. This is preceded by some lemmas and corollaries.
\begin{lemma}\label{p7}
If $r_{\pm}(z)\in H^{1}_{z}(\mathbb{R})$, then
\begin{align}
&\sup_{x\in\mathbb{R}}\|\mathcal{C}^{\pm}(r_{+}(z)e^{-2izx})\|_{L^{\infty}_{z}}\leq\frac{1}{\sqrt{2}}\|r_{+}\|_{H_{z}^{1}},\label{4.40}\\
&\sup_{x\in\mathbb{R}}\|\mathcal{C}^{\pm}(r_{-}(z)e^{2izx})\|_{L^{\infty}_{z}}\leq\frac{1}{\sqrt{2}}\|r_{-}\|_{H_{z}^{1}},\label{4.41}
\end{align}
and
\begin{align}
&\sup_{x\in\mathbb{R}^{+}}\|\langle x\rangle\mathcal{C}^{+}(r_{+}(z)e^{-2izx})\|_{L^{2}_{z}}\leq\|r_{+}\|_{H_{z}^{1}},\label{4.38}\\
&\sup_{x\in\mathbb{R}^{+}}\|\langle x\rangle\mathcal{C}^{-}(r_{-}(z)e^{-2izx})\|_{L^{2}_{z}}\leq\|r_{-}\|_{H_{z}^{1}}.\label{4.39}
\end{align}
\end{lemma}
\begin{proof}
We only give a detailed proof for $\mathcal{C}^{+}(r_{+}(z)e^{-2izx})$. The proofs for the remaining cases are similar.
Using the same method as for Proposition \ref{p9}, we can rewrite $\mathcal{C}^{+}(r_{+}(z)e^{-2izx})$ as
\begin{align*}
\begin{aligned}
\mathcal{C}^{+}\left(r_{+}(z)e^{-2izx}\right)(z)&=\lim_{\varepsilon\rightarrow0}\frac{1}{2\pi i}\int_{\mathbb{R}}\frac{r_{+}(s)e^{-2isx}}{s-(z+i\varepsilon)}ds \\
&=\lim_{\varepsilon\rightarrow0}\frac{1}{2\pi i}\int_{\mathbb{R}}\widehat{r}_{+}(\lambda)\int_{\mathbb{R}}\frac{e^{i(\lambda-2 x) s}}{s-(z+i \varepsilon)} ds d\lambda \\
&=\int_{2x}^{\infty}\widehat{r}_{+}(\lambda) e^{i(\lambda-2x)z}d\lambda
\end{aligned}
\end{align*}
according to the residue theorem and Jordan theorem.
Hence, we have
\begin{align*}
\sup_{x\in\mathbb{R}}\left\|\mathcal{C}^{+}\left(r_{+}(z) \mathrm{e}^{-2izx}\right)\right\|_{L_z^{\infty}} \leq\left\|\widehat{r}_{+}\right\|_{L^1} \leq\sqrt{\pi}\left\|\widehat{r}_{+}\right\|_{L^{2,1}} = \frac{1}{\sqrt{2}}\left\|r_{+}\right\|_{H_{z}^1}.
\end{align*}
Since $\widehat{r}_{+}\in L^{2,1}(\mathbb{R})$, we have
\begin{align*}
\sup_{x\in\mathbb{R}^{+}}\|\langle x\rangle\mathcal{C}^{+}(r_{+}(z)e^{-2izx})\|_{L^{2}_{z}}=&\sup_{x \in\mathbb{R}^{+}}\left\|\langle x\rangle \int_{2x}^{\infty}\widehat{r}_{+}(\lambda) \mathrm{e}^{i(\lambda-2x)z}d\lambda\right\|_{L_z^2}\\
\leq& \sqrt{2\pi}\left\|\widehat{r}_{+}\right\|_{L^{2,1}}=\left\|r_{+}\right\|_{H^1}
\end{align*}
by Proposition 1 in reference \cite{Pel}.
\end{proof}
\begin{corollary}\label{p7+}
If $r_{\delta,\pm}(z)\in H^{1}_{z}(\mathbb{R})$, then
\begin{align}
&\sup_{x\in\mathbb{R}}\|\mathcal{C}^{\pm}(r_{\delta,+}(z)e^{-2izx})\|_{L^{\infty}_{z}}\leq\frac{1}{\sqrt{2}}\|r_{\delta,+}\|_{H_{z}^{1}},\label{4.40+}\\
&\sup_{x\in\mathbb{R}}\|\mathcal{C}^{\pm}(r_{\delta,-}(z)e^{2izx})\|_{L^{\infty}_{z}}\leq\frac{1}{\sqrt{2}}\|r_{\delta,-}\|_{H_{z}^{1}},\label{4.41+}
\end{align}
and
\begin{align}
&\sup_{x\in\mathbb{R}^{-}}\|\langle x\rangle\mathcal{C}^{-}(r_{\delta,+}(z)e^{-2izx})\|_{L^{2}_{z}}\leq\|r_{\delta,+}\|_{H_{z}^{1}},\label{4.38+}\\
&\sup_{x\in\mathbb{R}^{-}}\|\langle x\rangle\mathcal{C}^{+}(r_{\delta,-}(z)e^{2izx})\|_{L^{2}_{z}}\leq\|r_{\delta,-}\|_{H_{z}^{1}}.\label{4.39+}
\end{align}
\end{corollary}
\begin{lemma}\label{l10}
If $r_{\pm}(z)\in H^{1}_{z}(\mathbb{R})$ and $r_{j}(k), j=1,2$ satisfy the condition \eqref{smallr}, then there exists a positive constant $C$ such that
\begin{align}
&\sup_{x\in\mathbb{R}^{+}}\|\langle x\rangle m_{-}^{(2)}(x,z)\|_{L_z^{2}}\leq C\|r_{-}\|_{H_{z}^{1}},\label{4.53}\\
&\sup_{x\in\mathbb{R}^{+}}\|\langle x\rangle n_{+}^{(1)}(x,z)\|_{L_z^{2}}\leq C\|r_{+}\|_{H_{z}^{1}}.\label{4.54}
\end{align}
If in addition, $r_{\pm}(z)\in L_{z}^{2,1}(\mathbb{R})$, then there exists another positive constant $C$ such that
\begin{align}
&\sup_{x\in\mathbb{R}}\|\partial_{x}m_{-}^{(2)}(x,z)\|_{L_z^{2}}\leq C(\|r_{+}\|_{H_{z}^{1}\cap L_{z}^{2,1}}+\|r_{-}\|_{H_{z}^{1}\cap L_{z}^{2,1}}),\label{4.55}\\
&\sup_{x\in\mathbb{R}}\|\partial_{x}n_{+}^{(1)}(x,z)\|_{L_z^{2}}\leq C(\|r_{+}\|_{H_{z}^{1}\cap L_{z}^{2,1}}+\|r_{-}\|_{H_{z}^{1}\cap L_{z}^{2,1}}),\label{4.56}
\end{align}
where the constant $C$ depends on $\|r_{\pm}\|_{H_{z}^{1}\cap L_{z}^{2,1}}$.
\end{lemma}
\begin{proof}
We define matrix $\tilde{M}(x,z):=[m_{-}(x,z)-e_{1}, n_{+}(x,z)-e_{2}]$ according to the results \eqref{4.38} and \eqref{4.39}, and
we know that
\begin{align}\label{4.49}
[m_{-}-e_{1}, n_{+}-e_{2}]=[\mathcal{C}^{-}(n_{+}r_{-}e^{2izx}), \mathcal{C}^{+}(m_{-}r_{+}e^{-2izx})]
\end{align}
by \eqref{4.27} and \eqref{MR}. Also, starting from the expression \eqref{3.23} for $R(x,z)$, we define the matrices
\begin{align}\label{4.51}
R_{-}(x,z):=\begin{bmatrix}0&0\\r_{-}(z)e^{2izx}&0\end{bmatrix},\ \ R_{+}(x,z):=\begin{bmatrix}0&r_{+}(z)e^{-2izx}\\0&0\end{bmatrix},
\end{align}
then
\begin{equation}\label{R+R}
R_{-}(x,z)+R_{+}(x,z)=(I-R_{+}(x,z))R(x,z).
\end{equation}
By calculation we can obtain
\begin{equation}\label{4.50}
\tilde{M}-\mathcal{C}^{+}(\tilde{M}R_{+})-\mathcal{C}^{-}(\tilde{M}R_{-})=F,
\end{equation}
where
\begin{align*}
F(x,z)=\begin{bmatrix}
0&\mathcal{C}^{+}(r_{+}(z)e^{-2izx})(z)\\ \mathcal{C}^{-}(r_{-}(z)e^{2izx})(z)&0
\end{bmatrix}.
\end{align*}
The left-hand side of the equation \eqref{4.50} can be rewritten as
\begin{align*}
\tilde{M}-\mathcal{C}^{+}(\tilde{M}R_{+})-\mathcal{C}^{-}(\tilde{M}R_{-})=&\tilde{M}-\tilde{M}R_{+}-\mathcal{C}^{-}(\tilde{M}R_{-}+\tilde{M}R_{+})\\
=&\tilde{M}(I-R_{+})-\mathcal{C}^{-}(\tilde{M}(I-R_{+})R)
\end{align*}
according to $\mathcal{C}^{+}-\mathcal{C}^{-}=I$ and \eqref{R+R}. We note that $\tilde{M}(x,z)(I-R_{+}(x,z))\triangleq J(x,z)$,  then \eqref{4.50} can be rewritten as
\begin{equation}\label{4.57}
J-\mathcal{C}^{-}(JR)=F,
\end{equation}
where
\begin{align*}
J(x,z)=\begin{bmatrix}m_{-}^{(1)}-1&n_{+}^{(1)}-(m_{-}^{(1)}-1)r_{+}e^{-2izx}\\ m_{-}^{(2)}&n_{+}^{(2)}-1-m_{-}^{(2)}r_{+}e^{-2izx}.
\end{bmatrix}.
\end{align*}

We denote the superscripts $(R_{1}), (R_{2})$ represent the first and second rows of the square matrix, respectively. Multiplying the matrix $\rho_{1}(k)$ on both sides of the equation \eqref{4.57} and considering the second row of the matrices give
\begin{equation}\label{R2}
(J\rho_{1})^{(R_{2})}-(\mathcal{C}^{-}(JR)\rho_{1})^{(R_{2})}=(F\rho_{1})^{(R_{2})}.
\end{equation}
The relation \eqref{3.27-} $R\rho_{1}=\rho_{1}S$ and the equation \eqref{tk1} yield
\begin{equation*}
(\mathcal{C}^{-}(JR)\rho_{1})^{(R_{2})}=\mathcal{C}^{-}((J\rho_{1})^{(R_{2})}S),
\end{equation*}
so \eqref{R2} can be rewritten as
\begin{equation}
(I-\mathcal{C}_{S}^{-})(J\rho_{1})^{(R_2)}=(F\rho_{1})^{(R_2)}.
\end{equation}
By the \eqref{4.15}, we know that for every $x\in\mathbb{R}$, there exists a positive constant $\eta$ such that
\begin{equation*}
\|(J\rho_{1})^{(R_2)}\|_{L_{z}^2}\leq \eta\|(F\rho_{1})^{(R_2)}\|_{L_{z}^2},
\end{equation*}
which implies
\begin{equation}\label{4.59}
\|m_{-}^{(2)}\|_{L_{z}^{2}}+\|2ik(n_{+}^{(2)}-1-m_{-}^{(2)}r_{+}e^{-2izx})\|_{L_{z}^{2}}\leq\eta\|\mathcal{C}^{-}(r_{-}e^{2izx})\|_{L_{z}^2}.
\end{equation}
From \eqref{4.39} we immediately get \eqref{4.53}, and in addition, from \eqref{smallr} we have
\begin{align}\label{4.61}
\|2ik(n_{+}^{(1)}-1)\|_{L_{z}^2}\leq (\eta+1)\|\mathcal{C}^{-}(r_{-}e^{2izx})\|_{L_z^2}.
\end{align}

Similarly, by multiplying the matrix $\rho_{2}(k)$ on both sides of the equation \eqref{4.57} and considering the first row of the matrices, we can obtain for every $x\in\mathbb{R}$, there exists
a positive constant $\eta$ such that
\begin{equation}\label{4.62}
\|\frac{1}{2ik}(m_{-}^{(1)}(x,z)-1)\|_{L_{z}^{2}}+\|n_{+}^{(1)}-(m_{-}^{(1)}-1)r_{+}e^{-2izx}\|_{L_{z}^{2}}\leq\eta \|\mathcal{C}^{+}(r_{+}e^{-2izx})\|_{L_{z}^{2}}.
\end{equation}
Further, from \eqref{smallr} we have
\begin{align*}
\|n_{+}^{(1)}\|_{L^2_{z}}&\leq\|\frac{1}{2ik}(m_{-}^{(1)}-1)2ikr_{+}e^{-2izx}\|_{L_z^2}+\eta\|\mathcal{C}^{+}(r_{+}e^{-2izx})\|_{L_z^{2}}\\
&\leq 2\eta\|\mathcal{C}^{+}(r_{+}e^{-2izx})\|_{L_z^2},
\end{align*}
using \eqref{4.38} we can get \eqref{4.54}.

We take the derivative of both sides of the equation \eqref{4.50} with respect to the variable $x$ to obtain
\begin{equation}\label{4.65}
\partial_{x}\tilde{M}-\mathcal{C}^{+}((\partial_{x}\tilde{M})R_{+})-\mathcal{C}^{-}((\partial_{x}\tilde{M})R_{-})=\hat{F},
\end{equation}
where
\begin{align*}
\hat{F}(x,z)=&\partial_{x}F+\mathcal{C}^{+}(\tilde{M}\partial_{x}R_{+})-\mathcal{C}^{-}(\tilde{M}\partial_{x}R_{-})\\
=&2i\begin{bmatrix}0&-\mathcal{C}^{+}(zr_{+}e^{-2izx})(z)\\ \mathcal{C}^{-}(zr_{-}e^{2izx})(z)&0\end{bmatrix}\\
&+2i\begin{bmatrix}\mathcal{C}^{-}(n_{+}^{(1)}zr_{-}e^{2izx})(z)&-\mathcal{C}^{+}((m_{-}^{(1)}-1)zr_{+}e^{-2izx})(z)\\
\mathcal{C}^{-}((n_{+}^{(2)}-1)zr_{-}e^{2izx})(z)&-\mathcal{C}^{+}(m_{-}^{(2)}zr_{+}e^{-2izx})(z)\end{bmatrix}.
\end{align*}
Similarly, we can rewrite \eqref{4.65} as
\begin{equation}
\hat{J}-\mathcal{C}^{-}(\hat{J}R)=\hat{F},
\end{equation}
where
\begin{align*}
\hat{J}(x,z)&:=\partial_{x}\tilde{M}(x,z)(I-R_{+}(x,z))\\
&=\begin{bmatrix}\partial_{x}m_{-}^{(1)}&-r_{+}e^{-2izx}\partial_{x}m_{-}^{(1)}+\partial_{x}n_{+}^{(1)}\\
\partial_{x}m_{-}^{(2)}&-r_{+}e^{-2izx}\partial_{x}m_{-}^{(2)}+\partial_{x}n_{+}^{(2)}
\end{bmatrix}.
\end{align*}
Repeating the same procedure as in the above proof and using \eqref{4.59},\eqref{4.62} and the Proposition \ref{l6} yields the results \eqref{4.55} and \eqref{4.56}.
\end{proof}

\begin{corollary}\label{l10+}
If $r_{\delta,\pm}(z)\in H^{1}_{z}(\mathbb{R})$ and $r_{j}(k), j=1,2$ satisfy the condition \eqref{smallr}, then there exists a positive constant $C$ such that
\begin{align}
&\sup_{x\in\mathbb{R}^{-}}\|\langle x\rangle m_{\delta,+}^{(2)}(x,z)\|_{L_z^{2}}\leq C\|r_{\delta,-}\|_{H^{1}},\label{4.53+}\\
&\sup_{x\in\mathbb{R}^{-}}\|\langle x\rangle n_{\delta,-}^{(1)}(x,z)\|_{L_z^{2}}\leq C\|r_{\delta,+}\|_{H^{1}}.\label{4.54+}
\end{align}
If in addition, $r_{\delta,\pm}(z)\in L_{z}^{2,1}(\mathbb{R})$, then there exists another positive constant $C$ such that
\begin{align}
&\sup_{x\in\mathbb{R}}\|\partial_{x}m_{\delta,+}^{(2)}(x,z)\|_{L_z^{2}}\leq C(\|r_{\delta,+}\|_{H_{z}^{1}\cap L_{z}^{2,1}}+\|r_{\delta,-}\|_{H_{z}^{1}\cap L_{z}^{2,1}}),\label{4.55+}\\
&\sup_{x\in\mathbb{R}}\|\partial_{x}n_{\delta,-}^{(1)}(x,z)\|_{L_z^{2}}\leq C(\|r_{\delta,+}\|_{H_{z}^{1}\cap L_{z}^{2,1}}+\|r_{\delta,-}\|_{H_{z}^{1}\cap L_{z}^{2,1}}),\label{4.56+}
\end{align}
where the constant $C$ depends on $\|r_{\delta,\pm}\|_{H_{z}^{1}\cap L_{z}^{2,1}}$.
\end{corollary}

\begin{proposition}\label{l11}
If $r_{\pm}(z)\in H_{z}^{1}(\mathbb{R})\cap L_{z}^{2,1}(\mathbb{R})$ and $r_{j}(k),\  j=1,2$ satisfy the condition \eqref{smallr}, then $u\in H^{2}(\mathbb{R})\cap H^{1,1}(\mathbb{R})$ satisfies
\begin{align}\label{4.73}
\|u\|_{H^{2}(\mathbb{R})\cap H^{1,1}(\mathbb{R})}\leq C(\|r_{+}\|_{ H_{z}^{1}(\mathbb{R})\cap L_{z}^{2,1}(\mathbb{R})}+\|r_{-}\|_{ H_{z}^{1}(\mathbb{R})\cap L_{z}^{2,1}(\mathbb{R})}),
\end{align}
where C is a constant depend on $\|r_{\pm}\|_{H_{z}^{1}\cap L_{z}^{2,1}}$. Moreover, the map
\begin{equation}\label{4.80}
[r_-(z),r_+(z)]\mapsto u
\end{equation}
is Lipschitz continuous from $H_{z}^{1}(\mathbb{R})\cap L_{z}^{2,1}(\mathbb{R})$ to $H^{2}(\mathbb{R})\cap H^{1,1}(\mathbb{R})$.
\end{proposition}
\begin{proof}
We first prove that $u\in L^{2,1}(\mathbb{R})$. The reconstruction formula \eqref{4.72} can be rewritten as
\begin{equation*}
u(x)e^{i\int^{+\infty}_{x}u(y)v(y)dy}=\frac{2}{\pi i}\int_{\mathbb{R}}r_{+}(z)e^{-2izx}dz+h_{1}(x),
\end{equation*}
where
\begin{equation*}
h_{1}(x):=\frac{2}{\pi i}\int_{\mathbb{R}}[m_{-}^{(1)}(x,z)-1]r_{+}(z)e^{-2izx}dz.
\end{equation*}
From \eqref{4.27} and \eqref{MR}, we have
\begin{align*}
h_{1}(x)=&\frac{2}{\pi i}\int_{\mathbb{R}}\mathcal{C}^{-}(n_{+}^{(1)}(x,z)r_{-}(z)e^{2izx})(z)r_{+}(z)e^{-2izx}dz\\
=&-\frac{2}{\pi i}\int_{\mathbb{R}}n_{+}^{(1)}(x,z)r_{-}(z)e^{2izx}\mathcal{C}^{+}(r_{+}e^{-2izx})(z)dz,
\end{align*}
and therefore
\begin{align*}
\sup_{x\in\mathbb{R}^{+}}|\langle x\rangle^{2}h_{1}(x)|
&\leq  \frac{2}{\pi} \|r_{-}\|_{L_{z}^{\infty}}\sup_{x\in\mathbb{R}^{+}}\|\langle x\rangle n_{+}^{(1)}\|_{L_{z}^{\infty}}\sup_{x\in\mathbb{R}^{+}}\|\langle x\rangle\mathcal{C}^{+}(r_{+}e^{-2izx})\|_{L_{z}^{2}}\\
&\lesssim_{\|r_{\pm}\|_{H^{1}}}\|r_{+}\|_{H^{1}}
\end{align*}
from \eqref{4.38} and \eqref{4.54}. Furthermore,
\begin{align*}
\|u\|_{L^{2,1}(\mathbb{R}^{+})}&=\frac{2}{\pi}\|\langle x\rangle\int_{\mathbb{R}}r_{+}(z)e^{-2izx}dz\|_{L^{2}(\mathbb{R}^{+})}+\|\langle x\rangle h_{1}(x)\|_{L^{2}(\mathbb{R}^{+})}\\
&\lesssim \|\langle x\rangle\widehat{r}_{+}\|_{L^{2}}+\sup_{x\in\mathbb{R}^{+}}|\langle x\rangle^{2}h_{1}(x)|\|\langle x\rangle^{-1}\|_{L^{2}}\\
&\lesssim_{\|r_{\pm}\|_{H^{1}}}\|r_{+}\|_{H^{1}}.
\end{align*}
Similarly, we can rewrite \eqref{4.91} as
\begin{equation*}
u(x)e^{i\int^{+\infty}_{x}u(y)v(y)dy}=\frac{2}{\pi i}\int_{\mathbb{R}}r_{\delta,+}(z)e^{-2izx}dz+h_{2}(x)
\end{equation*}
where
\begin{equation*}
h_{2}(x):=\frac{2}{\pi i}\int_{\mathbb{R}}[m_{\delta,+}^{(1)}(x,z)-1]r_{\delta,+}(z)e^{-2izx}dz.
\end{equation*}
Similar to the above proof we have
\begin{equation*}
\|u\|_{L^{2,1}(\mathbb{R}^{-})}\lesssim_{\|r_{\delta,\pm}\|_{H^{1}}}\|r_{\delta,+}\|_{H^{1}}\lesssim_{\|r_{\pm}\|_{H^{1}}}\|r_{+}\|_{H^{1}}+\|r_{-}\|_{H^{1}}
\end{equation*}
from \eqref{4.27}, \eqref{MR}, \eqref{4.38+}, \eqref{4.54+} and Proposition \ref{p9}. Thus we have completed the proof of $u\in L^{2,1}(\mathbb{R})$.

Next we prove that $u\in H^{2}(\mathbb{R})\cap H^{1,1}(\mathbb{R})$. By the same analytical technique we can obtain
\begin{align*}
&\|\partial_{x}[v(x)e^{\frac{1}{2i}\int^{+\infty}_{x}u(y)v(y)dy}]\|_{L^{2,1}(\mathbb{R}^{+})}\lesssim_{\|r_{\pm}\|_{H^{1}}}\|r_{-}\|_{H^{1}},\\
&\|\partial_{x}[v(x)e^{\frac{1}{2i}\int^{+\infty}_{x}u(y)v(y)dy}]\|_{L^{2,1}(\mathbb{R}^{-})}\lesssim_{\|r_{\pm}\|_{H^{1}}}\|r_{+}\|_{H^{1}}+\|r_{-}\|_{H^{1}}
\end{align*}
from the reconstruction formulas \eqref{4.71} and \eqref{4.90}. Because $H^{1}(\mathbb{R})$ is embedded into $L^{\infty}(\mathbb{R})$, we have
\begin{equation*}
\|u\|_{H^{1,1}(\mathbb{R})}\lesssim_{\|r_{\pm}\|_{H^{1}}}\|r_{+}\|_{H^{1}}+\|r_{-}\|_{H^{1}}.
\end{equation*}
We take the derivative of \eqref{4.71} and \eqref{4.90} with respect to the variable $x$, and using the same analysis we can get
\begin{equation*}
\|u\|_{H^{2}(\mathbb{R})}\lesssim_{\|r_{\pm}\|_{H^{1}\cap L^{2,1}}}\|r_{+}\|_{H^{1}\cap L^{2,1}}+\|r_{-}\|_{H^{1}\cap L^{2,1}}
\end{equation*}
by Lemma \ref{p7}, \ref{l10}, and the corollary \ref{p7+}, \ref{l10+}.

Finally, we prove that the map \eqref{4.80} is Lipschitz continuous. Suppose that $r_{\pm}, \tilde{r}_{\pm} \in H^{1}(\mathbb{R})\cap L^{2,1}(\mathbb{R})$ satisfy $\|r_{\pm}\|_{H^{1}\cap L^{2,1}}, \|\tilde{r}_{\pm}\|_{H^{1}\cap L^{2,1}}\leq \gamma$ for some $\gamma>0$. Denote the corresponding potentials by $u$ and $\tilde{u}$ respectively. Using the reconstruction formulas \eqref{4.72}, \eqref{4.71}, \eqref{4.91}, \eqref{4.90} and repeating the above proof we can obtain that there exists $\gamma$-dependent constant $C(\gamma)$ such that
\begin{equation*}
\|u-\tilde{u}\|_{H^{2}(\mathbb{R})\cap H^{1,1}(\mathbb{R})}\leq C(\gamma)(\|r_{+}-\tilde{r}_{+}\|_{H^{1}\cap L^{2,1}}+\|r_{-}-\tilde{r}_{-}\|_{H^{1}\cap L^{2,1}}).
\end{equation*}
\end{proof}
\section{Time evolution and global solutions}

Suppose that the fundamental vector solution $\phi(x,t,k)$ satisfying the Lax pair \eqref{1.9} and \eqref{1.10} associated with the potential $u(x,t)$ has the following form
\begin{equation*}
\varphi_{\pm}(x,t,k)e^{-ik^{2}x-2ik^{4}t}, \ \ \psi_{\pm}(x,t,k)e^{ik^{2}x+2ik^{4}t}.
\end{equation*}
For any fixed $t$, $\varphi_{\pm}(x,t,k)$ and $ \psi_{\pm}(x,t,k)$ has the same conclusion as the Corollary \ref{c2}, and satisfies the same boundary conditions
\begin{equation*}
\varphi_{\pm}(x,t,k)\rightarrow e_1,\ \
\psi_{\pm}(x,t,k)\rightarrow e_2,\ \ \text{as}\  x\rightarrow \pm\infty.
\end{equation*}
We define matrices
\begin{align*}
&J_{-}(x, t,k):=[\varphi_{-}(x,t,k),\psi_{-}(x,t,k)]e^{-ik^2\sigma_{3}x},\\
&J_{+}(x, t,k):=[\varphi_{+}(x,t,k),\psi_{+}(x,t,k)]e^{-ik^2\sigma_{3}x},
\end{align*}
that satisfy equation \eqref{1.9}, then by the theory of ODE it follows that there exists $\Lambda(t,k)=\begin{bmatrix}a(t,k)&c(t,k)\\b(t,k)&d(t,k)\end{bmatrix}$ such that $J_{-}(x,t,k)=J_{+}(x,t,k)\Lambda(t,k)$, i.e.
\begin{align}\label{5.1}
\begin{aligned}
&\begin{bmatrix}\varphi_-(x,t,k)e^{-ik^2x}&\psi_-(x,t,k)e^{ik^2x}\end{bmatrix}\\
=&\begin{bmatrix}\varphi_+(x,t,k)e^{-ik^2x}&\psi_+(x,t,k)e^{ik^2x}\end{bmatrix}\begin{bmatrix}a(t,k)&c(t,k)\\b(t,k)&d(t,k)\end{bmatrix}.
\end{aligned}
\end{align}
Thus a version of the scattering relation \eqref{2.59} with time $t$ is established.
And we can define the time-dependent reflection coefficients
\begin{equation}
r_1(t,k):=\frac{b(t,k)}{a(t,k)},\ \ r_2(t,k):=\frac{c(t,k)}{d(t,k)},\ \ k\in \mathbb{R}\cup i\mathbb{R}
\end{equation}
and
\begin{equation}
r_{-}(t,z):=2ik r_{1}(t,k),\ \ r_{+}(t,z):=-\frac{r_2(t,k)}{2ik},\ \ z\in\mathbb{R}.
\end{equation}

Some properties of the reflection coefficients are given in the following lemma.
\begin{lemma}\label{p10}
If $u_{0}\in H^{2}(\mathbb{R})\cap H^{1,1}(\mathbb{R})$ satisfies the condition \eqref{small+}, then for every $t\in\mathbb{R}$
\begin{equation}\label{5.3a-}
\|r_{j}(t,k)\|_{L_{z}^{2,1}\cap L_{z}^{\infty}}=\|r_{j}(0,k)\|_{L_{z}^{2,1}\cap L_{z}^{\infty}},\ \ j=1,2,
\end{equation}
and $\forall  t\in[0,T]$
\begin{equation}\label{5.3a}
\|r_{\pm}(t,z)\|_{H_{z}^{1}(\mathbb{R})\cap L_{z}^{2,1}(\mathbb{R})}\leq C(T)\|r_{\pm}(0,z)\|_{H_{z}^{1}(\mathbb{R})\cap L_{z}^{2,1}(\mathbb{R})},
\end{equation}
where $C(T)$ is a constant that depends on $T$ and is linear with respect to $T$.
\end{lemma}
\begin{proof}
We define the matrices
\begin{align*}
\tilde{J}_{-}(x,t,k):=J_{-}(x, t,k)e^{-2ik^{4}\sigma_{3}t},\ \
\tilde{J}_{+}(x,t,k):=J_{+}(x, t,k)e^{-2ik^{4}\sigma_{3}t},
\end{align*}
that satisfy the linear systems \eqref{1.9} and \eqref{1.10}. It follows from relation \eqref{5.1} that there exists $\tilde{\Lambda}(t,k)$ such that
\begin{equation*}
\tilde{J}_{-}(x,t,k)=\tilde{J}_{+}(x,t,k)\tilde{\Lambda}(t,k),
\end{equation*}
where
\begin{align*}
\tilde{\Lambda}(t,k)=e^{2ik^{4}\sigma_{3}t}\Lambda(t,k)e^{-2ik^{4}\sigma_{3}t}=\begin{bmatrix}a(t,k)&c(t,k)e^{4ik^4t}\\b(t,k)e^{-4ik^4t}&d(t,k)\end{bmatrix}.
\end{align*}
Substituting $\tilde{J}_{+}(x,t,k)\tilde{\Lambda}(t,k)$ into \eqref{1.10} and using the zero trace property of \eqref{1.9}, it follows that $\partial_{t}\tilde{\Lambda}(t,k)=0$, which means that
\begin{align*}
&a(t,k)=a(0,k),\ b(t,k)=b(0,k)e^{4ik^{4}t},\\ &d(t,k)=d(0,k),\ c(t,k)=c(0,k)e^{-4ik^{4}t}.
\end{align*}
So we have
\begin{align}
&r_{1}(t,k)=r_{1}(0,k)e^{4iz^{2}t},\ r_{2}(t,k)=r_{2}(0,k)e^{-4iz^{2}t},\label{5.3-}\\
&r_{-}(t,z)=r_{-}(0,z)e^{4iz^{2}t},\ r_{+}(t,z)=r_{+}(0,z)e^{-4iz^{2}t},\label{5.3}
\end{align}
which imply that $\forall t\in\mathbb{R}$,  $|r_{j}(t,k)|=|r_{j}(0,k)|,\ j=1,2$ and $|r_{\pm}(t,z)|=|r_{\pm}(0,z)|, z\in\mathbb{R}$.
Because of
\begin{align*}
\|\partial_{z} r_{\pm}(t,z)\|_{L_{z}^{2}}=&\|\partial_{z} r_{\pm}(0,z)\mp 8iztr_{\pm}(0,z)\|_{L_{z}^{2}}\\
\leq &\|\partial_{z} r_{\pm}(0,z)\|_{L_{z}^{2}}+8t\|\partial_{z} r_{\pm}(0,z)\|_{L_{z}^{2,1}},
\end{align*}
we have that the inequality \eqref{5.3a} holds.
\end{proof}

We denote $M(x,t,z)$ as the solution for the RH problem \ref{RH3} under time-dependent reflection coefficients $r_{\pm}(t,z)$. Since the properties \eqref{5.3a-} and \eqref{5.3a} of the time-dependent reflection coefficients, it follows from Proposition \ref{l9} that $M(x,t,z)$ exists uniquely. We denote $\Psi(x,t,z)$ as the unique solution for the RH problem \ref{RH1} under time-dependent reflection coefficients $r_{j}(t,k), \ k=1,2$. By inverse scattering theory \cite{Fok}, \cite{Liu1} it can be verified that the $\Psi(x,t,z)e^{-ik^2\sigma_{3}x-2ik^{4}\sigma_{3}t}$ satisfies the Lax pair \eqref{1.9} and \eqref{1.10}. Therefore the function $u(x,t)$ reconstituted from the reconstruction formulas is the solution of \eqref{1.1}. The following proposition shows that $u(x,t)$ can be controlled by $u_{0}$ when $t\in[0,T]$.

\begin{proposition}
If $u_{0}\in H^{2}(\mathbb{R})\cap H^{1,1}(\mathbb{R})$ satisfies the condition \eqref{small+}, then for every $t\in[0,T]$ we have
\begin{equation}\label{5.5}
\left\|u(\cdot,t)\right\|_{H^{2}(\mathbb{R})\cap H^{1,1}(\mathbb{R})}\leq C(T)\left\|u_{0}\right\|_{H^{2}(\mathbb{R})\cap H^{1,1}(\mathbb{R})},
\end{equation}
and the mapping
\begin{equation}\label{5.7}
H^{2}(\mathbb{R})\cap H^{1,1}(\mathbb{R})\ni u_{0}\mapsto u(x,t)\in C([0,T], H^{2}(\mathbb{R})\cap H^{1,1}(\mathbb{R}))
\end{equation}
is Lipschitz continuous.
\end{proposition}
\begin{proof}
It follows from Proposition \ref{l11}, Proposition \ref{l6} and Lemma \ref{p10} that
\begin{align*}
\left\|u(\cdot,t)\right\|_{H^{2}\cap H^{1,1}}&\leq C_{1}(\left\|r_{+}(t,z)\right\|_{H_{z}^{1}\cap L_{z}^{2,1}}+\left\|r_{-}(t,z)\right\|_{H_{z}^{1}\cap L_{z}^{2,1}})\\
&\leq C_{2}(T)(\left\|r_{+}(0,z)\right\|_{H_{z}^{1}\cap L_{z}^{2,1}}+\left\|r_{-}(0,z)\right\|_{H_{z}^{1}\cap L_{z}^{2,1}})\\
&\leq C(T)\left\|u_0\right\|_{H^{2}\cap H^{1,1}}.
\end{align*}
The proof of \eqref{5.5} is completed.

Before proving the Lipschitz continuity of mapping \eqref{5.7}, we show that $u(x,t)$ is continuous with respect to time $t$ in the sense of the norm $H^{2}(\mathbb{R})\cap H^{1,1}(\mathbb{R})$.
In fact, for any $t_{1},t_{2}\in [0,T]$
\begin{equation*}
\|r_{-}(t_{1},\cdot)-r_{-}(t_{2},\cdot)\|_{H^{1}(\mathbb{R})\cap L^{2,1}(\mathbb{R})}\lesssim\|r_{-}(0,\cdot)\|_{H^{1}(\mathbb{R})\cap L^{2,1}(\mathbb{R})}.
\end{equation*}
For any $\varepsilon>0$, there exists $N>0$ such that
\begin{equation*}
\left\|r_{-}(0,\cdot)\right\|_{H^{1}(|z|>N)\cap L^{2,1}(|z|>N)}<\varepsilon,
\end{equation*}
when $|z|\leq N$, we have
\begin{align*}
|e^{4iz^{2}(t_1-t_{2})}-1|&\leq 4N^{2}|t_1-t_2|,\\ |t_1e^{4iz^{2}t_1}-t_2e^{4iz^{2}t_2}|&\leq (4N^{2}T+1)|t_1-t_2|,
\end{align*}
and
\begin{align*}
&\left\|r_{-}(t_{1},\cdot)-r_{-}(t_{2},\cdot)\right\|_{H^{1}(|z|\leq N)\cap L^{2,1}(|z|\leq N)}\\
\leq & C(N,T)|t_1-t_2|\left\|r_{-}(0,\cdot)\right\|_{H^{1}(|z|\leq N))\cap L^{2,1}(|z|\leq N)},
\end{align*}
thus
\begin{equation*}
\left\|r_{-}(t_{1},\cdot)-r_{-}(t_{2},\cdot)\right\|_{H^{1}(\mathbb{R})\cap L^{2,1}(\mathbb{R})}<\varepsilon.
\end{equation*}
We can get similar result for $r_{+}(t,z)$, and
Proposition \ref{l11} gives
\begin{equation*}
\left\|u(\cdot,t_1)-u(\cdot,t_{2})\right\|_{H^{2}(\mathbb{R})\cap H^{1,1}(\mathbb{R})}<\varepsilon.
\end{equation*}

Let $u_{0}, \tilde{u}_{0} \in H^{2}(\mathbb{R})\cap H^{1,1}(\mathbb{R})$ satisfy $\left\|u_{0}\right\|_{H^{2}\cap H^{1,1}}, \left\|\tilde{u}_{0}\right\|_{H^{2}\cap H^{1,1}}\leq U$  for some $U>0$. Denote the corresponding scattering data by $r_{\pm}, \tilde{r}_{\pm}$ respectively. Estimates \eqref{5.3a} and Lipschitz continuity in Proposition \ref{l6} and Proposition \ref{l11} shows that
\begin{align*}
&\left\|u-\tilde{u}\right\|_{C([0,T], H^{2}\cap H^{1,1})}
= \left\|u(\cdot,t^{*})-\tilde{u}(\cdot,t^*)\right\|_{H^{2}\cap H^{1,1}}\\
\leq &C_{1}(U,T)(\|r_{+}(t^{*},\cdot)-\tilde{r}_{+}(t^{*},\cdot)\|_{H^{1}\cap L^{2,1}}+\|r_{-}(t^{*},\cdot)-\tilde{r}_{-}(t^{*},\cdot)\|_{H^{1}\cap L^{2,1}})\\
\leq & C_{2}(U,T)(\|r_{+}(0,\cdot)-\tilde{r}_{+}(0,\cdot)\|_{H^{1}\cap L^{2,1}}+\|r_{-}(0,\cdot)-\tilde{r}_{-}(0,\cdot)\|_{H^{1}\cap L^{2,1}})\\
\leq & C(U,T)\|u_{0}-\tilde{u}_{0}\|_{H^{2}\cap H^{1,1}},
\end{align*}
where $t^{*}\in [0,T]$, and $C(U,T)$ is a polynomial function with respect to $T$.
\end{proof}

Finally, we give the proof of the Theorem \ref{t1}.

\noindent\textbf{Proof of Theorem \ref{t1}:}
Let $T_{max}>0$ be the maximal existence time of the local solution $u(x,t)$. Suppose that $T_{max}<\infty$, then by the estimate \eqref{5.5}, there exists a constant $M$ and a time series ${t_{n}}\rightarrow T_{max}$ such that
\begin{equation*}
\|u(\cdot,t_{n})\|_{H^{2}\cap H^{1,1}}\leq M.
\end{equation*}
Then there exists $\delta>0$ and sufficiently large $n$ such that $t_n+\delta>T_{max}$ and the solution $u(x,t)$ exists on $[0,t_n+\delta]$, which contradicts the fact that $T_{max}$ is the maximal existence time.

Let $u_{0}, \tilde{u}_{0} \in H^{2}(\mathbb{R})\cap H^{1,1}(\mathbb{R})$ satisfy $\left\|u_{0}\right\|_{H^{2}\cap H^{1,1}}, \left\|\tilde{u}_{0}\right\|_{H^{2}\cap H^{1,1}}\leq U$ for some $U>0$. Define
\begin{align*}
d(u(x,t), \tilde{u}(x,t)):=\sum_{n=1}^{\infty}\frac{\|u-\tilde{u}\|_{n}}{2^{n}(1+\|u-\tilde{u}\|_{n})},
\end{align*}
where
\begin{equation*}
\left\|u(x,t)\right\|_{n}=\left\|u(x,t)\right\|_{C([0,n],H^{2}\cap H^{1,1})},\ \  n\in\mathbb{N}.
\end{equation*}
Lipschitz continuity \eqref{5.7} shows that
\begin{align*}
d(u(x,t), \tilde{u}(x,t))&\leq \sum_{n=1}^{\infty}\frac{C(U,n)\left\|u_{0}-\tilde{u}_{0}\right\|_{H^2\cap H^{1,1}}}{2^{n}(1+C(U,n)\left\|u_0-\tilde{u}_0\right\|_{H^2\cap H^{1,1}})}\\
&\leq \sum_{n=1}^{\infty}\frac{C(U,n)}{2^{n}}\left\|u_0-\tilde{u}_0\right\|_{H^2\cap H^{1,1}}\\
&\leq C(U)\left\|u_0-\tilde{u}_0\right\|_{H^2\cap H^{1,1}}.
\end{align*}
Thus the proof of Theorem \ref{t1} is achieved.

    \vspace{5mm}
    \noindent\textbf{Acknowledgements}

    This work is supported by  the National Natural Science
    Foundation of China (Grant No. 11671095,51879045).\vspace{2mm}

    \noindent\textbf{Data Availability Statements}

    The data which support the findings of this study are available within the article.\vspace{2mm}

    \noindent\textbf{Conflict of Interest}

    The authors have no conflicts to disclose.

\end{document}